\documentclass[a4paper]{amsart}
\pdfoutput=1

\usepackage{etex}
\usepackage{bbm}
\usepackage{pbox}
\usepackage{booktabs}
\usepackage{arydshln}
\usepackage{aliascnt}

\newcommand{\bk}{\mathbbm{k}}
\newcommand{\m}{\to}
\newcommand{\cB}{\mathcal{B}}

\newcommand{\cS}{\mathcal{S}}\newcommand{\cT}{\mathcal{T}}

\newcommand{\cX}{\mathcal{X}}
\newcommand{\cY}{\mathcal{Y}}

\usepackage{caption}
\newcommand{\gA}{\bold{A}}
\newcommand{\gB}{\bold{B}}
\newcommand{\gC}{\bold{C}}

\newcommand{\gE}{\bold{E}}

\newcommand{\gR}{\bold{R}}
\newcommand{\gS}{\bold{S}}
\newcommand{\gT}{\bold{T}}

\newcommand{\bC}{\mathbb{C}}
\newcommand{\bF}{\mathbb{F}}

\newcommand{\bL}{\mathbb{L}}
\newcommand{\bN}{\mathbb{N}}

\newcommand{\bQ}{\mathbb{Q}}

\newcommand{\bZ}{\mathbb{Z}}
\newcommand{\Z}{\mathbb{Z}}
\newcommand{\F}{\mathbb{F}}
\newcommand{\Q}{\mathbb{Q}}

\AtBeginDocument{%
	\def\MR#1{}
}

\newcommand{\RB}{]]}

\newcommand{\LB}{\mathbb{[[}}
\newcommand{\C}{\mathbb{C}}
\newcommand{\R}{\mathbb{R}}

\usepackage{lmodern}
\usepackage{booktabs}
\usepackage[dvipsnames,svgnames,x11names,hyperref]{xcolor}
\usepackage{mathtools,url,graphicx,verbatim,amssymb,enumerate,stmaryrd,nicefrac}
\usepackage[pagebackref,colorlinks,citecolor=blue,linkcolor=blue,urlcolor=blue,filecolor=blue]{hyperref}
\usepackage{microtype}
\usepackage[margin=1.43in]{geometry}

\usepackage{tikz, tikz-cd}
\usetikzlibrary{matrix,calc,arrows,patterns}

\numberwithin{thmcounter}{section}
\newaliascnt{thmauto}{thmcounter}

\newaliascnt{Defauto}{thmcounter}

\newaliascnt{exauto}{thmcounter}

\newaliascnt{lemauto}{thmcounter}

\newaliascnt{propauto}{thmcounter}

\newaliascnt{corauto}{thmcounter}

\newaliascnt{remauto}{thmcounter}

\newaliascnt{convauto}{thmcounter}

\newtheorem{atheorem}{Theorem}

\newtheorem*{ThmA'}{Theorem A'}
\newtheorem*{ThmB'}{Theorem B'}
\newtheorem*{ThmC'}{Theorem C'}

\newtheorem{theorem}[thmauto]{Theorem}
\newtheorem{lemma}[lemauto]{Lemma}
\newtheorem{proposition}[propauto]{Proposition}

\theoremstyle{definition}
\newtheorem{definition}[Defauto]{Definition}

\theoremstyle{remark}
\newtheorem{remark}[remauto]{Remark}
\newtheorem{example}[exauto]{Example}

\newcommand{\cat}[1]{\mathsf{#1}}
\newcommand{\mr}[1]{{\rm #1}}

\newcommand{\fS}{\mathfrak{S}}

\newcommand{\lra}{\longrightarrow}

\newcommand{\relo}{\,\mr{rel}^0\,}

\newcommand{\rk}{\mr{rk}}
\newcommand{\GL}{\mr{GL}}
\newcommand{\VIC}{\mr{VIC}}

\newcommand{\Tor}{\mr{Tor}}
\newcommand{\Hom}{\mr{Hom}}
\newcommand{\hAut}{\mr{hAut}}
\newcommand{\Mod}{\mr{Mod}}
\newcommand{\coker}{\mr{coker}}
\newcommand{\Std}{\mr{Std}}
\newcommand{\Specht}{\mr{Sp}}

\newcommand{\SL}{\mr{SL}}
\newcommand{\St}{\mr{St}}

\newcommand{\quot}[2]{{\raisebox{.2em}{$\scriptstyle #1\!\!$}\left/\raisebox{-.2em}{$\!\scriptstyle #2\!$}\right.}}

\title{Improved homological stability for certain general linear groups}

 \author{Alexander Kupers}
 \email{a.kupers@utoronto.ca}
 \address{University of Toronto\\
 Department of Mathematics\\
  1265 Military Trail\\
   Toronto, ON\\
    Canada, M1C 1A4}
  
\thanks{Alexander Kupers acknowledges the support of the Natural Sciences and Engineering Research Council of Canada (NSERC), as well as the Research Competitiveness Fund of the University of Toronto at Scarborough.}

 \author{Jeremy Miller}\thanks{Jeremy Miller was supported in part by NSF grant DMS-1709726 and a Simons Collaboration Grant.}
  \email{jeremykmiller@purdue.edu}
\address{Purdue University\\
Department of Mathematics \\
 	 150 North University \\
 	 West Lafayette IN, 47907 \\USA}
  
\author{Peter Patzt}\thanks{Peter Patzt was supported by the Danish National Research Foundation
	through the Copenhagen Centre for Geometry and Topology (DNRF151) and
	the European Research Council under the European Union’s Seventh
	Framework Programme ERC Grant agreement ERC StG 716424 - CASe, PI Karim
	Adiprasito.}
\email{ppatzt@ou.edu}
\address{University of Copenhagen, Centre for Geometry and Topology,
	Universitatsparken 5, 2100 Copenhagen, Denmark}
\address{University of Oklahoma, Department of Mathematics, 601 Elm
	Avenue, Norman, OK, 73019, USA}

\date{\today}

\begin{document}

\maketitle

\begin{abstract}We prove that the general linear groups of the integers, Gaussian integers, and Eisenstein integers satisfy homological stability of slope 1 when using $\bZ[\quot 12]$-coefficients and of slope $\quot{2}{3}$ when using $\Z$-coefficients.
 \end{abstract}

\section{Introduction} The (co)homology of arithmetic groups has many connections to number theory and topology. Here, we focus on general linear groups of a number ring $R$. Borel proved that these satisfy \emph{homological stability} with $\bQ$-coefficients \cite[Theorem 7.5]{Bor}: the \emph{stabilisation map}
\[H_d(\GL_{n-1}(R);\Q) \lra H_d(\GL_n(R);\Q),\]
induced by the inclusion $\GL_{n-1}(R) \to \GL_n(R)$, is an isomorphism for $n \gg d$. Moreover, he gave explicit ranges depending $R$; for example, for $R = \bZ$, he proved it is \emph{of slope} $\quot{1}{4}$, which means that the stabilisation map is an isomorphism when $d \leq \quot{n}{4}-c$ for some constant $c$. The slope was improved by Maazen and Van der Kallen to $\quot{1}{2}$ when using $\bZ$-coefficients \cite[Theorem 5.5]{Maazen} \cite[Theorem 4.11]{vdKallen}, and by Galatius, the first author, and Randal-Williams, to $\quot{2}{3}$ when using $\bZ[\quot{1}{N}]$-coefficients with $N$ depending on the number ring \cite[Section 18.2]{e2cellsI}. It is a conjecture that $\GL_n(\Z)$ satisfies homological stability of slope 1 when using $\bQ$-coefficients \cite[Page 1]{CP} \cite[Page 1]{CFPconj}. 

We prove the following generalisation of this for three well-behaved number rings: the integers, Gaussian integers, and Eisenstein integers.

\begin{atheorem}\label{thm:integers} Let $R$ be the integers, the Gaussian integers, or the Eisenstein integers. Then the stabilisation map
	\[H_d(\GL_{n-1}(R);\Z[\quot 12]) \lra H_d(\GL_n(R);\Z[\quot 12])\] is surjective for $d \leq n-1$ and an isomorphism for $d \leq n-2$.
\end{atheorem}

We prove a weaker range for a larger class of rings.

\begin{atheorem}\label{thm:Euclidean} Let $R$ be a Euclidean domain. Then the stabilisation map
	\[H_d(\GL_{n-1}(R);\Z[\quot 12]) \lra H_d(\GL_n(R);\Z[\quot 12])\]
	is surjective for $d\leq \quot23 (n-1)$ and an isomorphism for $d\leq \quot23(n-1)-1$.
\end{atheorem} 
 
We also prove a result with integral coefficients under a hypothesis on the units of $R$. 

\begin{atheorem}\label{thm:integersf2} Let $R$ be a Euclidean domain such that $R/2$ is generated as an abelian group by the image of the units of $R$. Then the stabilisation map
	\[H_d(\GL_{n-1}(R);\Z) \lra H_d(\GL_n(R);\Z)\]
is surjective for $d \leq \quot23 (n-1)$ and an isomorphism for $d\leq\quot23 (n-1)-1$. \end{atheorem} 

There are many examples of Euclidean domains with $R/2$ generated by the units of $R$: examples include those in \autoref{thm:integers}, as well as $\bZ[\sqrt{2}]$ and $R=\F[x]$ with $\F$ a field of characteristic not $2$. See \autoref{optimal} for a discussion of the extent to which the above results are optimal.

\subsection*{Twisted coefficients} Homological stability for general linear groups with certain families of twisted coefficients was first established by Dwyer \cite[Proposition 1.1]{DwyerTwisted} (see also \cite[Theorem 5.6]{vdKallen}). More generally, homological stability holds for coefficients in a \emph{polynomial functor} as in Randal-Williams--Wahl \cite[Definition 4.10]{RWW} (see \autoref{DefPolyn}). In addition to the improved range for constant coefficients of \autoref{thm:integers}, we obtain one for polynomial coefficients:
 
 \begin{atheorem}\label{thm:twistedStab}
Let $R$ be the integers, the Gaussian integers, or the Eisenstein integers. Let $V$ be a polynomial functor over $\Z[\quot12]$ of degree $r$ in ranks $>m$. Then \[H_{d}(\GL_{n-1}(R);V_{n-1}) \lra H_d(\GL_n(R);V_{n})\] is surjective for $d\leq n-\max(r,m)-1$ and an isomorphism for $d\leq n-\max(r,m)-2$.

\end{atheorem}

\subsection*{Acknowledgments} The first author thanks S\o ren Galatius and Oscar Randal-Williams, who had a large influence on the ideas in this paper. Much of this project was completed as part of the American Institute of Mathematics SQuaRE \emph{``Secondary representation stability.''} We thank AIM for their support. We also thank Rohit Nagpal and Jennifer Wilson who participated in this SQuaRE but declined co-authorship. We thank Manuel Krannich for helpful conversations regarding twisted coefficients, and Calista Bernard for comments on earlier version.

\tableofcontents

\section{$E_\infty$-algebras and their homology} 
\label{sec:einfty}
We recall some of machinery for $E_k$-algebras, following \cite[Section 3]{e2cellsIII}. We shall forego the usual exposition of cellular $E_k$-algebras; an interested reader may consult \cite[Section 2]{e2cellsII}, \cite[Section 2]{e2cellsIII}, or the entirety of \cite{e2cellsI}. It is in the latter that the results discussed in this section are proven, unless stated otherwise.

\subsection{General linear groups as an $E_\infty$-algebra}\label{sec:gl-einfty} Fix a commutative ring $R$. Our starting point is the groupoid $\cat{V}_R$ with objects the non-negative integers, and morphisms given by
\[\cat{V}_R(n,m) = \begin{cases} \GL_n(R) & \text{if $n=m$,} \\
\varnothing & \text{otherwise}.\end{cases}\]
This has a symmetric monoidal structure given as follows: on objects, the monoidal product is addition and on morphisms, it is block sum of matrices as a homomorphism $\GL_n(R) \times \mr{GL}_m(R) \to \GL_{n+m}(R)$. The symmetry is given by the identity on objects, and on morphisms by conjugation with the $(n + m) \times (n +m)$-matrix
\[T_{n,m} \coloneqq \begin{bmatrix} 0 & \mr{id}_n \\
\mr{id}_m & 0 \end{bmatrix}.\]
Let $\cat{sSet}$ denote the category of simplicial sets and $\cat{sSet}^{\cat{V}_R}$ the category of functors $\cat{V}_R \to \cat{sSet}$. The latter is simplicially enriched and admits a symmetric monoidal structure using Day convolution. As a consequence, we can make sense of $E_k$-algebras in $\cat{sSet}^{\cat{V}_R}$, which in this paper means algebras over the \emph{non-unital} little $k$-cubes operad (i.e.~the space of $0$-ary operations is empty). The category $\cat{sSet}^{\cat{V}_R}$ admits a projective model structure, which makes it into a simplicially enriched symmetric monoidal model category. Right transfer of it yields a model structure on the category $\cat{Alg}_{E_k}(\cat{sSet}^{\cat{V}_R})$ of $E_k$-algebras in $\cat{sSet}^{\cat{V}_R}$.

The functor given by $\underline{\ast}_{>0}(n) = \varnothing$ if $n=0$ and $\ast$ otherwise, is uniquely a non-unital commutative monoid in $\smash{\cat{sSet}^{\cat{V}_R}}$, so in particular a non-unital $E_\infty$-algebra. Taking its cofibrant replacement in the category $\cat{Alg}_{E_\infty}(\cat{sSet}^{\cat{V}_R})$ we obtain a non-unital $E_\infty$-algebra $\gT$ with a weak equivalence $\gT \to \underline{\ast}_{>0}$.

Let $\bN$ denote the symmetric monoidal groupoid with objects given by non-negative integers, only identity morphisms, and monoidal product given by addition (the symmetry is then uniquely determined). The functor $\rk \colon \cat{V}_R \to \bN$ given by the identity on objects is symmetric monoidal, and the restriction functor $\rk^* \colon \cat{Alg}_{E_\infty}(\cat{sSet}^\bN) \to \cat{Alg}_{E_\infty}(\cat{sSet}^{\cat{V}_R})$ participates in a Quillen adjunction with left Quillen functor $\rk_*$. We shall define 
\[\gR \coloneqq \rk_*(\gT) \simeq \bL \rk_*(\underline{\ast}_{>0}) \in \cat{Alg}_{E_\infty}(\cat{sSet}^\bN),\]
whose underlying object in $\cat{sSet}^\bN$ is described by
\[\gR(n) \simeq \begin{cases} \varnothing & \text{if $n=0$,} \\
B\mr{GL}_n(R) & \text{if $n>0$.}\end{cases}\]
Since we are interested in the homology of the values of $\gR$ with coefficients in certain commutative rings $\bk$, we may as well replace $\cat{sSet}$ with the category $\cat{sMod}_\bk$ of simplicial $\bk$-modules and correspondingly $\cat{sSet}^\bN$ by the category $\cat{sMod}_\bk^\bN$ of functors $\bN \to \cat{sMod}_\bk$. That is, we may study instead $\gR_\bk \coloneqq \bk[\gR] \in \cat{Alg}_{E_\infty}(\cat{sMod}_\bk^\bN)$.

\subsection{$E_1$- and $E_\infty$-homology}There is a homology theory for $E_\infty$-algebras $\gS$ in $\cat{sMod}_\bk^\bN$, obtained by taking the homology of the derived $E_\infty$-indecomposables $Q^{E_\infty}_\bL(\gS) \in \cat{sMod}_\bk^\bN$:
\[H_{n,d}^{E_\infty}(\gS) \coloneqq H_d(Q^{E_\infty}_\bL(\gS)(n)).\]
Any $E_\infty$-algebra can be considered as an $E_1$-algebra and has analogous $E_1$-homology groups defined in terms of the derived $E_1$-indecomposables as
\[H_{n,d}^{E_1}(\gS) \coloneqq H_d(Q^{E_1}_\bL(\gS)(n)).\]
These control the number of $E_\infty$- or $E_1$-cells in a minimal CW-approximation of $\gS$.

Since $\gR_\bk$ is defined as a derived pushforward along $\rk \colon \cat{V}_R \to \bN$ and the symmetric monoidal category $\cat{V}_R$ satisfies the properties of \cite[Lemma 3.2]{e2cellsIII}, we may compute the $E_1$-homology in terms of the $E_1$-splitting complex. From now on, we assume that $R$ is a principal ideal domain (PID); equivalently, it is a Dedekind domain with trivial class group.

\begin{definition}The \emph{$E_1$-splitting complex}\label{E1splitDef} $S^{E_1}_\bullet(R^n)$ is the semisimplicial set with $p$-simplices given by the ordered collections $(V_0,\ldots,V_{p+1})$ of non-zero proper summands of $R^n$ such that the natural map $V_0 \oplus \cdots \oplus V_{p+1} \to R^n$ is an isomorphism. The face map $d_i$ takes the sum of the $i$th and $(i+1)$st terms.\end{definition}

These are the non-degenerate simplices in the nerve of the poset $\cS^{E_1}(R^n)$ with objects pairs $(V_0,V_1)$ of non-zero proper summands of $R^n$ such that the natural map $V_0 \oplus V_1 \to R^n$ is an isomorphism, and $(V_0,V_1) \leq (V'_0,V'_1)$ if $V_0$ is a summand of $V_0'$ and $V_1'$ is a summand of $V_1$.
		
Let $S^{E_1}(R^n)$ denote its geometric realisation, which has an action of $\GL_n(R)$. It was proven by Charney \cite[Theorem 1.1]{Charney} that this is $(n-2)$-spherical and thus only its $(n-2)$st reduced homology group can be non-zero. This is a $\bZ[\GL_n(R)]$-module, and being a top homology group it is necessarily free as an abelian group.

\begin{definition}The \emph{$E_1$-Steinberg module of $R^n$} is the $\bZ[\GL_n(R)]$-module 
	\[\mr{St}^{E_1}(R^n) \coloneqq \widetilde{H}_{n-2}(S^{E_1}(R^n)).\]
\end{definition}

The following is a consequence of \cite[Remark 17.2.10]{e2cellsI}:

\begin{proposition}\label{prop:e1-splitting}
For $R$ a PID and $\bk$ a commutative ring, there is an isomorphism
\[H^{E_1}_{n,d}(\gR_\bk) \cong H_{d-n+1}(\GL_n(R);\mr{St}^{E_1}(R^n) \otimes \bk).\]
\end{proposition}
		
\section{The complex of partial frames and augmented partial frames} To eventually prove vanishing results for $E_1$-homology, we study the complexes of partial frames and of augmented partial frames in $R^n$. These complexes were introduced by Church--Putman \cite{CP} to study resolutions of Steinberg modules. The complex of partial frames is closely related to the complex of partial bases studied Maazen, van der Kallen, and others \cite{Maazen, vdKallen}.

\subsection{Definitions and connectivity results}

\subsubsection{Simplicial complexes} We presume the reader is familiar with simplicial complexes and their geometric realisation, as well as standard terminology such as simplices and links.

We say a simplicial complex is \emph{$d$-dimensional} or \emph{$d$-connected} if its geometric realisation is. It is \emph{$d$-spherical} if it is simultaneously $d$-dimensional and $(d-1)$-connected, in which case its geometric realisation is homotopy equivalent to a wedge of $d$-spheres. A simplicial complex is \emph{Cohen--Macauley} of dimension $d$ if it is $d$-spherical and the link of every $k$-simplex is $(d-k-1)$-spherical.

\subsubsection{Partial frames} \label{sec:partial-frames} Maazen proved his homological stability results by studying simplicial complexes of partial bases \cite{Maazen}, an approach which was extended by van der Kallen \cite{vdKallen}. We shall use a closely related simplicial complexes of partial frames, replacing basis vectors by the lines they span. From now on, we fix a PID $R$. 

\begin{definition}
	A vector $\vec v \in R^n$ is called \emph{primitive} if its span is a summand. In that case, we denote its span by $v$. Similarly if $v$ is a line, we let $\vec v$ denote a primitive vector which spans that line (which is well defined up to multiplication by a unit).\end{definition}

The term \emph{line} in this paper means a rank one free summand.

\begin{definition}A \emph{partial frame} is an unordered collection of lines $ v_0, \ldots  v_p$ such that there are lines $ v_{p+1},\ldots,  v_{n-1}$ such the map $v_0 \oplus \cdots \oplus  v_{n-1} \to R^n$ is an isomorphism. 
\end{definition}

\begin{definition} \label{complexesDef}
	The \emph{complex of partial frames} $B_n$ is the simplicial complex with $p$-simplices given by the set of partial frames of size $p+1$. A simplex $[w_0,\ldots,  w_q]$ is a face of $[v_0,\ldots,  v_p]$ if and only if $\{ w_0,\ldots,  w_q \} \subseteq \{ v_0,\ldots,  v_p \}$.
	\end{definition}

Viewing $R^m$ as the sub-module of $R^{m+n}$ spanned by the first $m$ standard basis vectors $\vec e_1,\ldots, \vec e_m$, we introduce the shorthand
\[B_n^m \coloneqq \mr{Link}_{B_{n+m}}(e_1,\ldots,  e_m).\]
Observe that $B_n^0 = B_n$. The following is \cite[Theorem 3.7]{KMPW}.

\begin{theorem} \label{thm:maazen}
	For $n \geq 1$ and $m \geq 0$ and $R$ is Euclidean, the simplicial complexes $B_n^m$ are Cohen--Macauley of dimension $(n-1)$.
\end{theorem}

\subsubsection{Augmented partial frames} Church--Putman \cite{CP} introduced a simplicial complex of augmented partial frames, by adding ``additive'' simplices to $B_n^m$.

\begin{definition}An \emph{augmented partial frame} is a collection of lines $ v_0,\ldots,  v_p$ such that, after reordering,  $ v_1, \ldots,  v_p$ is a partial frame and there are units $u_1,u_2$ with $\vec v_0=u_1 \vec v_1 + u_2 \vec v_2$ or in other words, there is a choice of representatives of lines such that $\vec v_0 = \vec v_1 + \vec v_2$.\end{definition}

\begin{definition}\label{def:augmented-frames} The \emph{complex of augmented partial frames} $BA_n$ is the simplicial complex with $p$-simplices given by the union of the set of partial frames of size $p+1$ and the set of augmented partial frames of size $p+1$. A simplex $[w_0,\ldots,  w_q]$ is a face of $[v_0,\ldots,  v_p]$ if and only if $\{ w_0,\ldots,  w_q \} \subseteq \{ v_0,\ldots,  v_p \}$.\end{definition}

We will need the following subcomplexes. 

\begin{definition} For $n>0$, let $BA_n^m$ denote the  subcomplex of $\mr{Link}_{BA_{n+m}}( e_1,\ldots,  e_m)$ of simplices $ v_0,\ldots,v_p$ such that $\vec v_i \notin R^m$ for all $0 \leq i \leq p$.\end{definition}
 
The following result is Church--Putman \cite[Theorems 4.2 and C']{CP} for $R$ the integers, and \cite[Theorem 3.16]{KMPW} for $R$ the Gaussian integers or the Eisenstein integers. 

\begin{theorem}[Church--Putman, K.--M.--P.--Wilson] \label{thm:cp}
	For $R$ the integers, Gaussian integers, Eisenstein integers, the simplicial complexes $BA_n$ and $BA_n^m$ are all Cohen--Macaulay. 
\end{theorem}

The dimension of $BA_n$ and $BA_n^m$ depends on $n+m$. If $n+m \leq 1$, then they are $(n-1)$-dimensional, otherwise they are $n$-dimensional. Thus the previous theorem says that $BA_n$ and $BA_n^m$ are $n$-spherical for $n+m \geq 2$ and hence $(n-1)$-connected. Following \cite[Definition 4.9]{CP}, we make the following definition.

\begin{definition}
	Let $\sigma$ be a simplex of $BA_n^m$ spanned by vertices $v_0,\ldots,v_p$.
	\begin{enumerate}
		\item We say $\sigma$ is an \emph{internally additive} simplex if $\vec v_i = \vec v_k + \vec v_j$ for some $i,j,k$ and choice of representatives of lines. 
		\item We say $\sigma$ is an \emph{externally additive} simplex if $\vec v_i =  \vec e_k + \vec v_j$ for some $i,j,k$ and choice of representatives of lines.  
	\end{enumerate} We say a simplex is \emph{additive} if it is internally additive or externally additive.
\end{definition}

\subsection{Generating set for the top homology} Since $B_n^m$ is $(n-1)$-spherical, the only possible non-zero reduced homology group is the $(n-1)$st. In this subsection, we will construct a generating set for this homology group. We follow the strategy used in \cite[Lemma 2.58]{MPP} in the case that $R$ is a field. We first introduce notation for certain spheres in the complex of partial frames.

\begin{definition} \label{defJoin}\
	\begin{enumerate}[(i)]
		\item Let $\{v_0,\ldots,v_p\}$ be a collection of vertices of $B^m_n$ such that any subset of size $p$ forms a $(p-1)$-simplex (but the entire set need not form a simplex). Then we denote by $\langle v_0,\ldots, v_p \rangle$ the full subcomplex of $B_n$ consisting of the vertices $v_0,\ldots ,v_p$ excluding $[v_0,\ldots,v_p]$ in the event that it forms a simplex. 
		\item Let $\{v_{j,0},\ldots, v_{j,p_j} \}$ with $1 \leq j \leq q$ be collections of vertices of $B^m_n$ such that a subset of \[ v_{1,0},\ldots, v_{1,p_1}, \ldots , v_{q,0},\ldots, v_{q,p_q} \] forms a simplex if it is a subset of a set of the form \[v_{1,0},\ldots, \widehat{ v_{1,i_1}},\ldots v_{1,p_1}, \ldots , v_{q,0},\ldots, \widehat{ v_{q,i_q}},\ldots v_{q,p_q}. \] 
		Then we denote by $\langle v_{1,0},\ldots, v_{1,p_1} \rangle * \cdots *\langle v_{q,0},\ldots, v_{q,p_q} \rangle$ the full subcomplex of $B_n$ consisting of the vertices $v_{1,0},\ldots, v_{1,p_1}, \ldots, v_{q,0},\ldots, v_{q,p_q}$, excluding simplices containing simplices of the form $[v_{j,0},\ldots, v_{j,p_j} ]$ in the event that it forms a simplex.
	\end{enumerate}
\end{definition}

Note that $\langle v_0,\ldots, v_p \rangle \cong S^{p-1}$ and $\langle v_{1,0},\ldots v_{1,p_1}\rangle* \cdots *\langle v_{q,0},\ldots, v_{q,p_q} \rangle \cong S^{-1+{\scriptstyle \sum}_j p_j}$. 

\begin{definition}
	Let $\{v_{j,0},\ldots, v_{j,p_j} \}$ for $1 \leq j \leq q$ be collections of vertices of $B_n^m$ satisfying the condition of \autoref{defJoin} (ii). Let 
	\[ \LB v_{1,0},\ldots v_{1,p_1}\RB * \cdots *\LB v_{q,0},\ldots, v_{q,p_q}\RB \in \widetilde H_{-1+{\scriptstyle \sum}_j p_j}(B_n^m)\] denote the image of the fundamental class of $\langle v_{1,0},\ldots, v_{1,p_1} \rangle * \cdots * \langle v_{q,0},\ldots, v_{q,p_q} \rangle$, where the sphere is oriented by the order of the lines.
\end{definition}

When we write $\LB v_{1,0},\ldots, v_{1,p_1} \RB * \cdots *\LB v_{q,0},\ldots, v_{q,p_q} \RB$,
we implicitly assume that the conditions of \autoref{defJoin} (ii) are satisfied. These satisfy the following relation: permuting the vertices within a list $v_{j,0},\cdots,v_{j,p_j}$ yields the same class up to the sign of the permutation. For example,  \begin{align*}& \LB v_{1,0},v_{1,1},v_{1,2} \ldots v_{1,p_1} \RB * \cdots *\LB v_{q,0},\ldots, v_{q,p_q}\RB \\
&\qquad =-\LB v_{1,1}v_{1,0},v_{1,2},\ldots v_{1,p_1}\RB * \cdots *\LB v_{q,0},\ldots, v_{q,p_q}\RB.\end{align*}

The goal of this subsection is to prove the following, which generalises \cite[Lemma 2.58]{MPP} in the case of fields.

\begin{theorem} \label{thmGen}
Let $R$ be a Euclidean domain such that $BA_n^m$ is spherical for all $n$ and $m$. Then for all $n \geq 1$ the group $\widetilde H_{n-1}(B_n^m)$ is generated by elements of the following form:
	\[ \LB v_1,v_2, w_2 \RB*\LB v_3,v_4, w_4 \RB* \cdots * \LB v_{2d-1},v_{2d},w_{2d}  \RB * \LB v_{2d+1},u_{2d+1}  \RB * \cdots * \LB v_n, u_n \RB   \] with \begin{enumerate}
		\item $e_1,\ldots,e_m,v_1,\ldots,v_n$ a frame,
		\item for all $i$, $\vec w_{2i}= \vec v_{2i-1} + \vec v_{2i} $ for some choice of representatives of lines, 
		\item for all $j$, $\vec u_j=\vec v_j + \vec v_i$ for some $i<j$ or $\vec u_j=\vec v_j + \vec e_i$ for some $i \leq m$ and for some choice of representatives of lines.
	\end{enumerate}
\end{theorem}

\begin{remark}
	When the set of generators in \autoref{thmGen} is empty, we mean that the group $\widetilde H_{n-1}(B^m_n) \cong 0$. This occurs only when $n=1$ and $m=0$. On the other hand, for $n=0$, $B_n^m$ is empty and $\widetilde H_{-1}(B_0^m) \cong \bZ$ which we view as being generated by the empty product.
\end{remark}

\begin{proof}[Proof of \autoref{thmGen}]
	We will prove the claim by strong induction on $n$. This requires the introduction of additional initial cases. For $n=0$, the complexes are empty, as explained in the previous remark, and we shall take the empty product as a generator (in contrast with the case $n=1$ and $m=0$). For $n=-1$, they are not defined, and we shall take them to have no generators.
	
	Assume we have proven the claim for all $d<n$ and all $m$. Since $BA_1^0$ is just a point, $\widetilde H_0(BA_1^0) \cong 0$. Thus, there is nothing to prove unless $n+m \geq 2$ so we shall also assume this from now on. Consider the exact sequence of a pair \[ H_{n}(BA_n^m,B_n^m)  \lra \widetilde H_{n-1}(B_n^m) \lra \widetilde H_{n-1}(BA_n^m)\] coming from the natural inclusion $B_n^m \to BA_n^m$. Since we assume $n+m \geq 2$,  \autoref{thm:cp} implies that $\widetilde H_{n-1}(BA_n^m) \cong 0$.  Thus it suffices to find a generating set for $H_{n}(BA_n^m,B_n^m)$ since it surjects onto the group of interest.
	
	Let $C_*$ denote the cellular chain functor and $\partial$ the differential. The group $C_p(BA_n^m,B_n^m)$ is generated by augmented partial bases $[v_0,\ldots ,v_p]$. The order on these vectors only matters up to sign. Let $\mathcal I_p$ denote the set of internally additive $p$-simplices of $BA_n^m$ and $\mathcal E_p$ denote the set of externally additive $p$-simplices of $BA_n^m$. Reorder the lines and pick representatives so that $\vec v_0= \vec v_1 + \vec v_2$ or $\vec v_0= \vec v_1 + \vec e_i$. Note that \[ \partial([v_0,\ldots ,v_p])=\sum_{i=3}^{i=p} (-1)^p [v_0,\ldots, \hat v_i ,\ldots ,v_p ] \qquad \text{for $[v_0,\ldots ,v_p] \in \mathcal I_p$},\]  
	as the first three terms in the sum defining the differential of cellular chains vanish in relative chains. Similarly, \[ \partial([v_0,\ldots ,v_p])=\sum_{i=2}^{i=p} (-1)^p [v_0,\ldots, \hat v_i ,\ldots ,v_p ] \qquad \text{ for $[v_0,\ldots ,v_p] \in \mathcal E_p$.}\]
	
	This gives isomorphism of chain complexes of $C_*(BA_n^m,B_n^m)$ with
	\begin{align*}&\left (\bigoplus_{\{v_0,v_1,v_2\} \in \mathcal I_2}  C_{*-3}(\mr{Link}_{BA_n^m}(v_0,v_1,v_2)) \right) \oplus  \left (\bigoplus_{\{v_0,v_1\} \in \mathcal E_1}  C_{*-2}(\mr{Link}_{BA_n^m}(v_0,v_1)) \right)\\
	&\quad \cong \left (\bigoplus_{\{v_0,v_1,v_2\} \in \mathcal I_2}  C_{*-3}(\mr{Link}_{B_n^m}(v_0,v_1)) \right) \oplus  \left (\bigoplus_{\{v_0,v_1\} \in \mathcal E_1}  C_{*-2}(\mr{Link}_{B_n^m}(v_0)) \right).\end{align*}
	In particular, $\widetilde H_{n-1}(BA_n^m,B_n^m)$ is isomorphic to
	\[\left (\bigoplus_{\{v_0,v_1,v_2\} \in \mathcal I_2}  \widetilde H_{*-4}(\mr{Link}_{B_n^m}(v_0,v_1)) \right) \oplus  \left (\bigoplus_{\{v_0,v_1\} \in \mathcal E_1} \widetilde H_{*-3}(\mr{Link}_{B_n^m}(v_0)) \right).\]
	
	Note that $\mr{Link}_{B_n^m}(v_0,v_1) \cong B_{n-2}^{m+2}$ and that $\mr{Link}_{B_n^m}(v_0) \cong B_{n-1}^{m+1}$ so by induction we know generating sets for their top degree reduced homology groups. Therefore, to finish the proof, it suffices to understand the connecting homomorphism \[ \delta \colon H_{n}(BA_n^m,B_n^m)  \lra \widetilde H_{n-1}(B_n^m).\] Let $[\alpha] \in H_n(BA_n^m,B_n^m)$ and let $\alpha \in \widetilde C_n(BA_n^m)$ be a chain representing $[\alpha]$. By definition, $\delta([\alpha])=[\partial(\alpha)]$, viewed as an element of $\widetilde H_{n-1}(B_n^m)$. 
	
	Let $\{b_0, b_1,b_2\} \in \mathcal I_2$ and let \[ \LB v_1,v_2, w_2 \RB*\LB v_3,v_4, w_4 \RB* \cdots * \LB v_{2d-1},v_{2d},w_{2d}  \RB *\LB v_{2d+1},u_{2d+1}  \RB * \cdots * \LB v_{n-3}, u_{n-3} \RB \] be a generator of $\mr{Link}_{B_n^m}(b_0,b_1)$. Its image in $\widetilde H_{n-1}(BA_n^m)$ under the composition of isomorphisms of the previous paragraph and the connecting homomorphism is \[  \LB  b_0,b_1, b_2\RB  *  \LB v_1,v_2, w_2 \RB* \cdots *  \LB v_{2d-1},v_{2d},w_{2d}   \RB * \LB v_{d+1},u_{d+1}   \RB * \cdots *  \LB v_{n-3}, u_{n-3}  \RB. \] Note that $\vec b_0= \vec b_1 + \vec b_2$ for some choice of representatives of lines so this class is of the form indicated in the statement of the theorem. Similarly, let $\{a_0,a_1\} \in \mathcal E_2$ and let \[  \LB v_1,v_2, w_2  \RB* \LB v_3,v_4, w_4  \RB* \cdots *  \LB v_{2d-1},v_{2d},w_{2d}   \RB * \LB v_{2d+1},u_{2d+1}   \RB * \cdots *  \LB v_{n-2}, u_{n-2}  \RB \] be a generator of $\mr{Link}_{B_n^m}(a_0)$. Its image in $\widetilde H_{n-1}(BA_n^m)$ under the composition of isomorphisms of the above paragraph and the connecting homomorphism is \[   \LB a_0,a_1 \RB \ast \LB  v_1,v_2, w_2  \RB * \cdots *  \LB v_{2d-1},v_{2d},w_{2d}   \RB * \LB v_{d+1},u_{d+1}   \RB * \cdots *  \LB v_{n-3}, u_{n-3}  \RB. \] We may shuffle the term $ \LB a_0,a_1 \RB$ to the right at the cost of a sign. Note that $\vec a_0= \vec a_1 + \vec e_i$ for some $i \leq m$ and choice of representatives of lines so this class is of the form indicated in the statement of the theorem. This completes the proof. 
\end{proof}

\begin{remark}
	Observe that the reason that in condition (3) in the statement we must allow $\vec u_j =\vec v_j + \vec v_i$ instead of just $\vec u_j =\vec v_j + \vec e_i$, is that the isomorphism $\mr{Link}_{B_n^m}(v_1) \cong B_{n-1}^{m+1}$ interchanges the roles of some of the lines $v_\alpha$ and $e_\beta$. 
\end{remark}

\subsection{Vanishing of coinvariants} A generating set of a $G$-representation $M$ can be used to compute the coinvariants $M_G$. Using the results of the previous subsection we will do so for $M = \widetilde{H}_{n-1}(B_n^m;\Z[\quot12])$ and $G = \GL(R^{m+n}, \text{ fix $R^m$})$ the subgroup of $\GL_{n+m}(R)$ fixing $R^m = \mr{span}\{\vec e_1,\ldots,\vec e_m\}$ pointwise.

\begin{theorem}\label{thm:top-partial}	
Let $R$ be a Euclidean domain such that $BA_n^m$ is spherical for all $n$ and $m$ (e.g. the integers, Gaussian integers, or Eisenstein integers). Let $n+m \geq 2$, and $n \geq 1$. Then $\widetilde H_{n-1}(B_n^m;\bZ[\quot12])_{\GL(R^{m+n}, \text{ fix $R^m$})} \cong 0$. 
\end{theorem}

\begin{proof}
	By  \autoref{thmGen},  $\widetilde H_{n-1}(B_n^m)$ is generated by elements of the form \[  \LB v_1,v_2, w_2  \RB * \LB v_3,v_4, w_4  \RB * \cdots *  \LB v_{2d-1},v_{2d},w_{2d}   \RB  * \LB v_{d+1},u_{d+1}   \RB  * \cdots *  \LB v_n, u_n  \RB    \]  satisfying the conditions listed in the theorem. The conditions $n+m \geq 2$ and $n \geq 1$ guarantee that we do not allow the empty product. Thus, it suffices to show the images of generators which are nonempty products vanish in $\widetilde H_{n-1}(B_n^m;\bZ[\quot12])_{\GL(R^{m+n}, \text{ fix $R^m$})}$. It suffices to discuss two cases. 
	\vspace{.5em}
	
	\noindent \textbf{Case 1: $2d=n$.}
Note $d \geq1 $ by assumption. Pick representatives such that $\vec w_2=\vec v_1+\vec v_2$. Let $g \in \GL(R^{m+n},\text{ fix $R^m$})$ have the property that $g(\vec v_1)=\vec v_2$, $g(\vec v_2)=\vec v_1$, $g(\vec v_i)=\vec v_i$ for $i>2$, and $g(\vec e_i) =\vec e_i$ for $i \leq m$. Then $g(v_1)=v_2$, $g(v_2)=v_1$ and $g(w_2)=w_2$. We have \[g( \LB v_1,v_2,w_2 \RB )= \LB v_2,v_1,w_2 \RB =- \LB v_1,v_2,w_2 \RB .\]  Thus, \begin{align*}&g \Big(  \LB v_1,v_2, w_2  \RB * \LB v_3,v_4, w_4  \RB * \cdots *  \LB v_{2d-1},v_{2d},w_{2d}   \RB   \Big) \\
	&\quad =  \LB v_2,v_1, w_2  \RB * \LB v_3,v_4, w_4  \RB * \cdots *  \LB v_{2d-1},v_{2d},w_{2d}   \RB  \\
	&\quad =- \LB v_1,v_2, w_2  \RB * \LB v_3,v_4, w_4  \RB * \cdots *  \LB v_{2d-1},v_{2d},w_{2d}   \RB  .\end{align*} Since $2$ is a unit in $\bZ[\quot12]$, this class vanishes in $\widetilde H_{n-1}(B_n^m;\bZ[\quot12])_{\GL(R^{m+n}, \text{ fix $R^m$})}$.
	\vspace{.5em}
	
	\noindent \textbf{Case 2: $2d<n$.}	Pick vector representatives so that $\vec u_n=\vec v_n + \vec v_j$ or $\vec u_n=\vec v_n + \vec e_j$. Let $\vec a=\vec u_n - \vec v_n$. Let $g \in \GL(R^{m+n},\text{ fix $R^m$})$ have the property that $g(\vec v_n)=-\vec v_n- \vec a$, $g(\vec v_i) =\vec v_i$ for $i <n$, and $g(\vec e_i) =\vec e_i$ for $i \leq m$. Then $g(v_n)=u_n$ and $g(u_n)=v_n$ and so 
	 \[g(  \LB  v_n,u_n \RB )= \LB u_n,v_n \RB =- \LB v_n,u_n \RB . \] Thus, \begin{align*}&g \Big(  \LB v_1,v_2, w_2  \RB * \LB v_3,v_4, w_4  \RB * \cdots *  \LB v_{2d-1},v_{2d},w_{2d}   \RB  * \LB v_{d+1},u_{d+1}   \RB  * \cdots *  \LB v_n, u_n  \RB  \Big) \\
	&\quad =  \LB v_1,v_2, w_2  \RB * \LB v_3,v_4, w_4  \RB * \cdots *  \LB v_{2d-1},v_{2d},w_{2d}   \RB  * \LB v_{d+1},u_{d+1}   \RB  * \cdots *  \LB u_n, v_n  \RB  \\
	&\quad =- \LB v_1,v_2, w_2  \RB * \LB v_3,v_4, w_4  \RB * \cdots *  \LB v_{2d-1},v_{2d},w_{2d}   \RB  * \LB v_{d+1},u_{d+1}   \RB  * \cdots *  \LB v_n, u_n  \RB .\end{align*} Since $2$ is a unit in $\bZ[\quot12]$, this class vanishes in $\widetilde H_{n-1}(B_n^m;\Z[\quot12])_{\GL(R^{m+n}, \text{ fix $R^m$})}$.
\end{proof}

\begin{proposition} \label{BPID}
Let $R$ be a PID. Then $\widetilde H_{0}(B_1^1;\bZ[\quot12])_{\GL(R^{2}, \text{ fix $R^1$})} \cong 0$. 
\end{proposition}

\begin{proof} The complex $B_1^1$ is the link of $e_1$ in $B_2$. It is discrete with vertices spans of lines of the form $r \vec e_1+\vec e_2$ with $r \in R$ (not necessarily a unit). Denote such a line by $v_r$ and observe that classes of the form $\LB e_2,v_r \RB$ generate $\widetilde H_{0}(B_1^1)$. The claim follows from the computation $\begin{bsmallmatrix}
1 & -r \\
0 & -1
\end{bsmallmatrix} \LB e_2,v_r \RB= \LB v_r,e_2 \RB=-\LB e_2,v_r \RB$.
\end{proof}

\section{$E_1$-homology} In  \autoref{sec:gl-einfty}, we defined an $E_\infty$-algebra $\gR_\bk$ in $\cat{sMod}^\bN_\bk$ for each commutative ring $\bk$ and PID $R$. In this section, we will prove that (c.f. \cite[Section 18.2]{e2cellsI})
\[\widetilde{H}^{E_1}_{n,d}(\gR_\bk) = 0 \qquad \text{for $d < n-1$},\] and that for $R$ the integers, the Gaussian integers, or the Eisenstein integers, we have
\[\widetilde{H}^{E_1}_{n,d}(\gR_{\bZ[\quot12]}) = 0 \qquad \text{for $(n,d) = (n,n-1)$ when $n \geq 2$.}\]
The arguments in this section follow closely those in \cite{e2cellsIII}, with \autoref{thm:top-partial} as novel input.

\subsection{Poset techniques} In our proof, we use several map-of-posets arguments. The first is due to van der Kallen and Looijenga \cite[Corollary 2.2]{vdKallenLooijenga}. Using this, one can prove the second as an adaption of a result of Quillen, cf.~\cite[Theorem 4.1]{e2cellsIII}. Recall that if $(\cX,\leq)$ is a poset and $x \in \cX$ then we have a poset $\cX_{<x} \coloneqq \{x' \in  X \setminus \{x\} \mid x' < x\}$ with the induced ordering, and similarly for $\cX_{> x}$. Given a map $f \colon \cX \to \cY$ of posets, we also have a poset $f_{\leq y} \coloneqq \{x \in X \mid f(x) \leq y\}$ and similarly for $f_{\geq y}$. The following results require the lengths of chains in a poset are bounded: this is satisfied when the poset is finite-dimensional, which is the case for all posets in this paper.

\begin{proposition}\label{prop:map-of-posets}Suppose that $\cX$ and $\cY$ have an upper bound to lengths of chains. Let $f \colon \cX \to \cY$ be a map of posets with the property that for some $n \in \bZ$ and function $t \colon \cY \to \bZ$, one of the following is true:
	\begin{itemize}
		\item $\cY_{<y}$ is $(t(y)-2)$-connected and $\cX_{f \geq y}$ is $(n-t(y)-1)$-connected, or
		\item $\cY_{>y}$ is $(n-t(y)-2)$-connected and $\cX_{f \leq y}$ is $(t(y)-1)$-connected.
	\end{itemize}
	Then $f$ is $n$-connected.
\end{proposition}

\begin{theorem}\label{thm:nerve-spherical} Let $f \colon \cX \to \cY$ be a map of posets which have an upper bound to lengths of chains, $n \in \bZ$, and $t_\cY \colon \cY \to \bZ$ be a function. Assume that 
	\begin{enumerate}[(i)]
		\item $\cY$ is $n$-spherical, 
		\item for every $y \in \cY$, $f_{\leq y}$ is $(n-t_\cY(y))$-spherical, and
		\item $\cY_{>y}$ is $(t_\cY(y)-1)$-spherical. \end{enumerate}
	Then $\cX$ is $n$-spherical and there is a canonical filtration $0 = F_{n+1} \subseteq F_n \subseteq \cdots \subseteq F_{-1} = \widetilde{H}_n(\cX)$ such that 
	\[F_{-1}/F_0 \cong \widetilde{H}_n(\cY), \quad \text{and} \quad F_{q}/F_{q+1} \cong \bigoplus_{t_\cY(y) = n-q} \widetilde{H}_{n-q-1}(\cY_{>y}) \otimes \widetilde{H}_{q}(\cX_{f\leq y}).\]
\end{theorem}

\subsection{Tits buildings and relative Tits buildings} We will use the complex of partial frames to study the $E_1$-splitting complexes using Tits buildings as intermediaries. Throughout this section we fix a PID $R$.

\begin{definition}\
	\begin{enumerate}[(i)]
		\item The \emph{Tits building $\cT(R^n)$} is the poset of proper non-trivial summands of $R^n$, ordered by inclusion. 
		\item For a summand $W \subseteq R^n$, the \emph{relative Tits building $\cT(R^n \relo W)$} is the subposet of $\cT(R^n)$ of those summands $V$ such that $V \cap W = \{0\}$ and $V+W$ is a summand.
	\end{enumerate}
\end{definition}

\subsubsection{Steinberg coinvariants}\label{sec:tits}
The Tits building $\cT(R^n)$ for a PID is isomorphic to the similarly defined Tits building for its field of fractions, cf.~\cite[Remark 2.8]{MPWY}. 

\begin{theorem}[Solomon--Tits] For $R$ a PID, $\cT(R^n)$ is $(n-3)$-connected.\end{theorem}

Since $\cT(R)$ is $(n-2)$-dimensional, the only possible non-zero reduced homology group is the $(n-2)$st. This is the so-called \emph{Steinberg module}
\[\St(R) \coloneqq \widetilde{H}_{n-2}(\cT(R^n);\Z),\]
which is a $\bZ[\GL_n(R)]$-module that is free as an abelian group. Lee--Szczarba \cite[Theorem 1.3]{LS} proved the following theorem about its coinvariants.

\begin{theorem}[Lee--Szczarba] For $R$ a Euclidean domain and $n \geq 2$, $\St(R^n)_{\GL_n(R)} = 0$. \end{theorem}

\subsubsection{Relative Steinberg coinvariants}\label{sec:relative-tits} 

We now prove the corresponding results for the relative Tits buildings $\cT(R^{m+n} \relo R^m)$ under an assumption which is satisfied when $R$ is the integers, Gaussian integers, or Eisenstein integers.

 We shall see the relative Tits buildings are $(n-2)$-connected and $(n-1)$-dimensional, and define the \emph{relative Steinberg modules} as 
\[\St(R^{m+n} \relo R^m) \coloneqq \widetilde{H}_{n-1}(\cT(R^{m+n} \relo R^m);\Z).\]
These are $\bZ[\GL(\R^{m+n},\text{ fix $R^m$})]$-modules which are free as abelian groups.

\begin{lemma} 
Let $R$ be a Euclidean domain. For $n \geq 0$ and $m \geq 1$, $\cT(R^{m+n} \relo R^m)$ is $(n-2)$-connected. If $BA_n^m$ is spherical for all $n$ and $m$ (e.g. the integers, Gaussian integers, or Eisenstein integers), then for $n \geq 1$ and $m \geq 1$, we also have \[\St(R^{m+n} \relo R^m)_{\GL(R^{m+n},\text{ fix $R^m$})} = 0.\]\end{lemma}

\begin{proof}The proof of the first part is by induction over $n$, and for each individual step we will deduce the second part about vanishing of coinvariants from  \autoref{thm:top-partial}. In the initial case $n=0$ there is nothing to prove.
	
	Recall from  \autoref{sec:partial-frames} that $B_n$ denotes the complex of partial frames of $R^n$, and $B^m_n$ the link in $B_{m+n}$ of $\{e_1,\ldots,e_m\} \subseteq R^{m+n}$. Let $\cB_n$ and $\cB^m_n$ be the associated posets of simplices. There is a $\GL(R^{m+n},\text{ fix $R^m$})$-equivariant map 
	\begin{align*} \mr{span} \colon \cB^m_n &\lra \cT(R^{m+n} \relo R^m) \\
	\{v_0,\ldots,v_k\} &\longmapsto \mr{span}(v_0,\ldots,v_k).\end{align*}
	
	For the induction step, we apply the second version of  \autoref{prop:map-of-posets} to this map, with the goal of proving that it is $(n-1)$-connected, and taking $t(V) = \rk(V)-1$. On the one hand, $\cY_{>V} \cong \cT(R^{m+n}/V \relo R^m)$, which is $(n-\rk(V)-2) = (n-1-t(V)-2)$-connected by the inductive hypothesis. On the other hand, $\cX_{f \leq V} \cong \cB_{\rk(V)}$, which is $(\rk(V)-2) = (t(V)-1)$-connected by  \autoref{thm:maazen}. Thus $\mr{span}$ is $(n-1)$-connected; since $\cB^m_n$ is $(n-2)$-connected by  \autoref{thm:maazen} the connectivity result follows, and since
	\[\widetilde H_{n-1}(B^m_n;\bZ[\quot12]) \lra \widetilde{H}_{n-1}(\cT(R^{m+n} \relo \mr{R}^m);\bZ[\quot12])\]
	is surjective, the second part follows from right-exactness of $\GL(R^{m+n},\text{ fix $R^m$})$-coinvariants and  \autoref{thm:top-partial}.
\end{proof}

\subsection{Split Steinberg coinvariants} The following argument closely resembles that of \cite{e2cellsIII}, whose notation and structure we have adapted. Even though the argument is essentially identical, we repeat it here to make this paper more self-contained.

Recall in the paragraph following \autoref{E1splitDef} we defined a poset $\cS^{E_1}(R^n)$ of splittings $(V_0,V_1)$ of $R^n$. We now define the following subposet.

\begin{definition} Let $W \subseteq R^n$ be a non-zero summand. Let $\cS^{E_1}(\cdot,W \subseteq \cdot \mid R^n)$ be the subposet of $\cS^{E_1}(R^n)$ of splittings $(V_0,V_1)$ with $W \subseteq V_1$.\end{definition}

We will often denote $\cS^{E_1}(R^n)$ by $\cS^{E_1}(\cdot,\cdot \mid R^n)$. In analogy with the definition of $\St^{E_1}(R^n)$, we define $\St^{E_1}(R^n \relo W)$ to be $\widetilde H_{n-\rk W-1}(\cS^{E_1}(\cdot,W \subseteq \cdot \mid R^n) )$ whenever the complex is spherical.

\begin{theorem}\label{thm.steinbergcoinv} Let $R$ be a Euclidean domain such that $BA_n^m$ is spherical for all $n$ and $m$ (e.g. the integers, Gaussian integers, or Eisenstein integers). Let $W \subseteq R^n$ be a non-zero summand and suppose $n \geq 1$. Then we have that:
	\begin{align*}
		\cS^{E_1}(\cdot,\cdot \mid R^n) &\text{ is $(n-2)$-spherical,}\\
		\cS^{E_1}(\cdot,W \subseteq \cdot \mid R^n) &\text{ is $(n-\rk W-1)$-spherical.}\end{align*}
	Moreover, for all $n \geq 2$ it is true that:
	\begin{align*}
		\left(\St^{E_1}(R^n) \otimes \bZ[\quot12]\right)_{\GL(R^n)} &= 0, \\
		\left(\St^{E_1}(R^n \relo W) \otimes \bZ[\quot12]\right)_{\GL(R^n \text{, fix } W)}& =0.\end{align*}
\end{theorem}

\noindent It shall be helpful to name four facts that are used in the proof of  \autoref{thm.steinbergcoinv}, proven in \hyperref[sec:tits]{Sections \ref{sec:tits}} and \ref{sec:relative-tits}:
\begin{enumerate}[\indent (1)]
	\item \label{enum.titsspherical} The Tits building $\cT(R^n)$ is $(n-2)$-spherical.
	\item \label{enum.reltitsspherical} The relative Tits building $\cT(R^n \relo W)$ is $(n-\dim W-1)$-spherical.
\end{enumerate}

\begin{enumerate}[\indent (1)$^\mr{St}$]
	\item \label{enum.steinbergcoinv} For $n \geq 2$, $(\mr{St}(R^n) \otimes \bZ[\quot12])_{\mr{GL}(R^n)} = 0$.
	\item \label{enum.relsteinbergcoinv} For $n \geq 2$, $(\St(R^n \relo W) \otimes \bZ[\quot12])_{\mr{GL}(R^n,\text{ fix $W$})} = 0$.
\end{enumerate}

\noindent Without loss of generality we can change the basis of $R^n$ so that $W=R^w$ spanned by the first $w$ basis vectors. The proof of  \autoref{thm.steinbergcoinv} will be an upwards induction over $(n,w)$ in lexicographic order, where $w=0$ is to be interpreted as the absolute version.

\subsubsection{The base cases} The base cases are those with $n \leq 1$; then there is nothing to prove. 

\subsubsection{The first reduction} 

\begin{proposition}\label{prop:first-reduction-general} 
	The case $(n,0)$ is implied by the cases $\{(n,w) \, | \, w >0\}$.\end{proposition}

\begin{proof}
	Consider the map of posets
	\begin{align*} \cX \coloneqq \cS^{E_1}(\cdot,\cdot|R^n) &\overset{f}\lra \cY \coloneqq \cT(R^n)^\mr{op} \\
	(A,B) &\longmapsto B,
	\end{align*}
	and for $U \in \cY$ let $t_\cY(U) = \dim(U)-1$. We verify the assumptions of  \autoref{thm:nerve-spherical}: 
	\begin{enumerate}[(i)]
		\item $\cY = \cT(R^n)^\mr{op}$ is $(n-2)$-spherical by (\ref{enum.titsspherical}).
		\item $\cY_{>U} = \cT(U)^\mr{op}$ is $(\dim(U)-2) = (t_\cY(U)-1)$-spherical by (\ref{enum.titsspherical}).
		\item $\cX_{f \leq U}=\cS^{E_1}(\cdot,U\subseteq \cdot|R^n)$. By the inductive hypothesis this is $(n-\dim(U)-1) = (n-2-t_\cY(U))$-spherical. 
	\end{enumerate}
	 \autoref{thm:nerve-spherical} then implies that $\cX=\cS^{E_1}(\cdot,\cdot|R^n)$ is $(n-2)$-spherical, and that there is a filtration of the $\bZ[\quot12][\mr{GL}(R^n)]$-module $\mr{St}^{E_1}(R^n) \otimes \bZ[\quot12]$ with $F_{-1}/F_0 = \mr{St}(R^n) \otimes \bZ[\quot12]$ and the other filtration quotients $F_q/F_{q+1}$ given by
	\begin{align*}
	&\bigoplus_{\substack{0 < U < R^n\\ \dim(U) = n-q-1}} \mr{St}(U) \otimes \widetilde{H}_{q}(S^{E_1}(\cdot,U \subseteq \cdot|R^n);\Z) \otimes \Z[\quot12]\\
	&\qquad\cong \mr{Ind}^{\mr{GL}(R^n)}_{\mr{GL}(R^n,\text{ pres $R^{n-q-1}$})} \left(\mr{St}(R^{n-q-1}) \otimes \St^{E_1}(R^n \relo R^{n-q-1}) \otimes \bZ[\quot12]\right),
	\end{align*}
	where $R^{n-q-1} \subseteq R^n$ is spanned by the first $(n-q-1)$ basis vectors and $\mr{GL}(R^n,\text{ pres $R^{n-q-1}$})$ is the subgroup of $\GL_n(R)$ that preserves $R^{n-q-1}$ as a subspace. It then suffices to prove that the covariants of each of these filtration quotients vanish.
	
	Firstly, for $q=-1$ we take $\mr{GL}(R^n)$-coinvariants of $\mr{St}(R^n) \otimes \bZ[\quot12]$ which vanish for $n \geq 2$ by (\ref{enum.steinbergcoinv})$^\mr{St}$. Secondly, for $q\geq 0$ we can use Shapiro's lemma to identify the $\mr{GL}(R^n)$-coinvariants of the induced module with the $\mr{GL}(R^n,\text{ pres $R^{n-q-1}$})$-coinvariants of 
	\[\mr{St}(R^{n-q-1}) \otimes \St^{E_1}(R^n \relo R^{n-q-1}) \otimes \bZ[\quot12].\]
	These vanish if the $\mr{GL}(R^n, \text{ fix $R^{n-q-1}$})$-coinvariants vanish, which follows from the inductive hypothesis as $n \geq 2$.
\end{proof}

\subsubsection{Cutting down} The following generalises an argument due to Charney \cite[p.\ 4]{Charney}. Given summands $V$ and $W$ of $P$, let $\cS^{E_1}(\cdot \subseteq V,W \subseteq \cdot|P)$ be the subposet of $\cS^{E_1}(\cdot,\cdot|P)$ consisting of those splittings $(A,B)$ such that $A \subseteq V$ and $W \subseteq B$.

\begin{lemma}\label{lem:cutting-down} Let $R$ be a PID. Let $V$ and $W$ be summands of $P$ such that $V \cap W = 0$ and $V+W$ is a proper summand. Let $C$ be a complement to $W$ in $P$ containing $V$. Then there is an isomorphism
	\[\cS^{E_1}(\cdot \subseteq V,W \subseteq \cdot|P) \cong \cS^{E_1}(\cdot \subseteq V,\cdot|C),\]
equivariantly for the subgroup of $\mr{GL}(P)$ preserving the summands $V$, $W$ and $C$.\end{lemma}

\begin{proof}Mutually inverse functors are given by
	\begin{align*}\cS^{E_1}(\cdot \subseteq V,W \subseteq \cdot|P) & \overset{\cong}{\longleftrightarrow} \cS^{E_1}(\cdot \subseteq V,\cdot|C) \\
	(A,B) & \longmapsto (A,B \cap C) \\
	(A',B' \oplus W) & \longmapsfrom (A',B') \end{align*} 
	This is well-defined, as it cannot happen that $B \cap C = \{0\}$, because $V \oplus W \neq P$ implies that $B \cap C$ has dimension at least $1$.\end{proof}

\subsubsection{The second reduction} 

\begin{proposition}\label{prop:second-reduction-general} 
	The case $(n,w)$ is implied by the cases $\{(n', w')\, | \, n' < n\}$.
\end{proposition}  

\begin{proof} 
	When $w = n-1$, $\cS^{E_1}(\cdot,R^w \subseteq \cdot|R^n)$ is isomorphic to $\cT(R^n \relo R^w)$ and we already know the result. So let us assume that $w \leq n-2$. 
	
	Consider the map  of posets
	\begin{align*}\cX \coloneqq \cS^{E_1}(\cdot,R^w \subseteq \cdot|R^n) &\overset{f}\lra \cY \coloneqq \cT(R^n \relo R^w) \\
	(A,B) &\longmapsto A, \end{align*}
	and take $t_\cY(V) = n-\rk V-w$. We verify the assumptions of  \autoref{thm:nerve-spherical}:
	\begin{enumerate}[(i)]
		\item $\cY = \cT(R^n \relo R^w)$ is $(n-w-1)$-spherical by (\ref{enum.reltitsspherical}).
		\item $\cY_{>V}$ is isomorphic to $\cT(R^n/V \relo R^w)$, which is $(n - \rk V - w - 1) = (t_\cY(V)-1)$-spherical by (\ref{enum.reltitsspherical}). 
		\item $\cX_{f \leq V}$ is $\cS^{E_1}(\cdot \subseteq V,R^w \subseteq \cdot|R^n)$, which we claim is $(\rk V-1)= (n-w-1-t_\cY(V))$-spherical.
		
		There are two cases: the first case is that $V \oplus R^w = R^n$, and then the complex is contractible (the splitting $(V,R^w)$ is terminal) and hence is a wedge of no $(n-w-1-t_\cY(V))$-spheres. The second case is that $V \oplus R^w \neq R^n$. Taking a complement $C$ to $R^w$ which contains $V$, we can apply the cutting down \autoref{lem:cutting-down} to see it is isomorphic to $\cS^{E_1}(\cdot \subseteq V,\cdot|C)$. Letting $C^\vee$ denote $\mr{Hom}(C,R)$, dualising shows this is isomorphic to $\cS^{E_1}(\cdot,V^\circ \subseteq \cdot|C^\vee)$, where $V^\circ = \{ f \in C^\vee \mid f(v) = 0 \text{ for all }v\in V\}$ is the annihilator. By induction this is $(n-w-(n-w-\rk V)-1)=(\rk (V)-1)$-spherical.
	\end{enumerate} 
	 \autoref{thm:nerve-spherical} then implies that $\cX = \cS^{E_1}(\cdot,R^w \subseteq \cdot|R^n)$ is $(n-w-1)$-spherical, and gives a filtration of the $\Z[\quot12][\mr{GL}(R^n, \text{ fix }R^w)]$-module $\St^{E_1}(R^n \relo R^w) \otimes \bZ[\quot12]$ with $F_{-1}/F_0 = \St(R^n \relo R^w) \otimes \bZ[\quot12]$  and further filtration quotients $F_q/F_{q+1}$ given by
	\begin{align*}&\bigoplus_{\substack{0 < V < R^n \\ V \cap R^w = \{0\},\\
	\text{$V+R^w$ a summand} \\ \rk V=q-1}} \St(R^n/V \relo R^w) \otimes \widetilde{H}_{q-2}(\cS^{E_1}(\cdot \subseteq V,R^w \subseteq \cdot|R^n)) \otimes \bZ[\quot12] \\
	&\qquad \cong \mr{Ind}^{\mr{GL}(R^n,\text{ fix $R^w$})}_{\mr{GL}(R^n,\text{ fix $R^w$, pres $V$})}
	 \begin{pmatrix}
	\St^{E_1}(R^n/V \relo R^w) \otimes
	\\
	\widetilde{H}_{q-2}(\cS^{E_1}(\cdot \subseteq V,R^w \subseteq \cdot|R^n);\Z) \otimes \bZ[\quot12]
	\end{pmatrix}
	\end{align*} 
	for $V$ a choice of $(q-1)$-dimensional summand complementary to $R^w$ (for example, the span of the last $q-1$ basis vectors). As before, it suffices to prove that $\mr{GL}(R^n,\text{ fix $R^w$})$-coinvariants of each of these filtration quotients vanish.
	
	The filtration quotient $F_{-1}/F_0 = \St^{E_1}(R^n \relo R^w) \otimes \bZ[\quot12]$ has trivial $\mr{GL}(R^m,\text{ fix $R^w$})$-coinvariants by (\ref{enum.relsteinbergcoinv})$^\mr{St}$. For the remaining filtration quotients, we may assume that $V \oplus R^w \neq R^n$ (otherwise the complex is contractible). By Shapiro's lemma the coinvariants are equal to the $\mr{GL}(R^n, \text{ fix $R^w$, pres $V$})$-coinvariants of
	\[ \St^{E_1}(R^n/V \relo R^w) \otimes \widetilde{H}_{q-2}(\cS^{E_1}(\cdot \subseteq V,R^w \subseteq \cdot|R^n);\Z) \otimes \bZ[\quot12].\] 
	To show that this vanishes in a range of degrees, we define a group $K$ by the extension
	\[0 \lra K \lra \mr{GL}(R^n, \text{fix $R^w$, pres $V$}) \lra \mr{GL}(R^n/V, \text{fix $R^w$}) \lra 0,\]
	that is, $K=\mr{GL}(R^n, \text{fix $R^w$, pres $V$, fix $R^n/V$})$. This acts trivially on $\St^{E_1}(R^n/V \relo R^w)$, so it suffices to show the vanishing of the $K$-coinvariants of
	\[\widetilde{H}_{\dim V-1}(\cS^{E_1}(\cdot \subseteq V,R^w \subseteq \cdot|R^n);\Z)\otimes \bZ[\quot12].\]
	
\noindent If $C$ is a complement to $R^w$ which contains $V$, then each element of $K$ has to preserve $C$ because $K$ acts as the identity on $R^n/V$. Thus  \autoref{lem:cutting-down} gives a $\Z[\quot12][K]$-module isomorphism
	\[\widetilde{H}_{\dim V-1}(\cS^{E_1}(\cdot \subseteq V,R^w \subseteq \cdot|R^n);\Z) \otimes \Z[\quot12] \cong \widetilde{H}_{\dim V-1}(\cS^{E_1}(\cdot \subseteq V,\cdot|C);\Z) \otimes \Z[\quot12].\]
	Furthermore the natural map $K \to \mr{GL}(C)$ is an isomorphism onto its image, given by the subgroup $\mr{GL}(C, \text{ pres $V$, fix $C/V$})$. Thus we are reduced to showing that the $\mr{GL}(C, \text{pres $V$, fix $C/V$})$-coinvariants of $\mr{St}^{E_1}(\cdot \subseteq V,\cdot|C)) \otimes \Z[\quot12]$ vanish. Dualising identifies this with the $\mr{GL}(C^\vee, \text{fix $V^\circ$})$-coinvariants of $\St^{E_1}(C^\vee \relo V^\circ) \otimes \Z[\quot12]$ which vanish in the required range by induction as $2 \leq \rk C < n$ (the lower bound follows as the annihilator $V^\circ \neq 0$ and $V \oplus R^w \neq R^n$).
\end{proof}

\begin{proposition} 
If $R$ a Euclidean domain then $(\St^{E_1}(R^2) \otimes \bZ[\quot12])_{\GL(R^2)} = 0 $.
\label{S2}
\end{proposition}

\begin{proof}
The proof is identical to the proof of \autoref{thm.steinbergcoinv} using \autoref{BPID} instead of \autoref{thm:top-partial}.
\end{proof}

\section{Proofs of \autoref{thm:integers} and \autoref{thm:Euclidean}} In this section we prove \autoref{thm:integers} and \autoref{thm:Euclidean} by combining a general homological stability result with the results of the previous section.

\subsection{A criterion for homological stability of slope $\quot{m}{m+1}$ after inverting $2$} We will use a general homological stability result for $E_k$-algebras. Fix a commutative ring $\bk$ and let $\gA \in \cat{Alg}_{E_k}(\cat{sMod}_{\bk}^\bN)$ be a non-unital $E_k$-algebra. We can unitalise and rectify it to obtain a unital associative algebra denoted $\overline{\gA} \in \smash{\cat{Alg}_{E_k^+}(\cat{sMod}_\bk^\bN)}$ \cite[Section 12.2.1]{e2cellsI}. Fixing a representative $\sigma \colon S^{1,0}_\bk \to \gA$ of $\sigma \in H_{1,0}(\gA)$, we define the quotient $\overline{\gA}$-module $\overline{\gA}/\sigma$ as the homotopy cofiber of
\[S^{1,0}_\bk \otimes \overline{\gA} \xrightarrow{\sigma \otimes \mr{id}} \overline{\gA} \otimes \overline{\gA} \overset{\mu}\lra \overline{\gA}\]
with $\mu$ the multiplication of $\overline{\gA}$. Its homology fits into a long exact sequence
\[\cdots \lra H_{g-1,d}(\overline{\gA}) \xrightarrow{\sigma \cdot -} H_{g,d}(\overline{\gA}) \lra H_{g,d}(\overline{\gA}/\sigma) \lra H_{g-1,d-1}(\overline{\gA}) \lra \cdots,\]
so that homological stability for $\gA$ is equivalent to a vanishing result for $H_{*,*}(\overline{\gA}/\sigma)$.

\begin{proposition} \label{E3}
Assume $\bk$ is a commutative ring with $2$ a unit in $\bk$, and $\gA \in \cat{Alg}_{E_k}(\cat{sMod}^\bN_\bk)$ with $k \geq 3$ and $H_{*,0}(\overline{\gA}) \cong \bk[\sigma]$ for $\sigma \in H_{1,0}(\gA)$. If $\smash{H^{E_1}_{n,d}}(\gA)=0$ for $d<n-1$, and 
\[H^{E_1}_{2,1}(\gA)=H^{E_1}_{3,2}(\gA)= \dots = H^{E_1}_{m+1,m}(\gA)=0,\]
then $H_{n,d}(\overline{\gA}/\sigma) = 0$ for $d \leq \quot{m}{m+1}(n-1)$. In particular, if $H^{E_1}_{n,d}(\gA)=0$ for $d<n$, then $H_{n,d}(\overline{\gA}/\sigma) = 0$ for $d < n$.
\end{proposition}

\begin{proof} By the universal coefficient theorem, it suffices to prove the result for $\bk=\F_p$ with $p$ odd or $\bk=\Q$, c.f.~Reduction 2 and 3 of \cite[Theorem 18.2]{e2cellsI}. We will describe the case $\bk=\F_p$ as the case $\bk=\Q$ is similar and easier. 
 
Using bar spectral sequences as in \cite[Theorem 14.2.1]{e2cellsI}, the $E_k$-indecomposables satisfy the same vanishing range as the $E_1$-indecomposables: 
	\[H^{E_k}_{n,d}(\gA)  =0 \text{ for $d<n-1$ and } H^{E_k}_{2,1}(\gA)=H^{E_k}_{3,2}(\gA)= \ldots = H^{E_k}_{m+1,m}(\gA)  =0.\]
 
A map $\sigma \colon S^{1,0}_{\bF_p} \to \gA$ representing a generator $\sigma$ of $H_{1,0}(\gA;\bF_p)$ yields a map
	\[f \colon \gE_k(S^{1,0}_{\F_p}) \lra \gA\]
of $E_k$-algebras. Consider now the long exact sequence of $E_k$-homology
\[\cdots \lra H^{E_k}_{n,d+1}(f) \lra H^{E_k}_{n,d}(\gE_k(S^{1,0}_{\F_p})) \lra H^{E_k}_{n,d}(\gA) \lra H^{E_k}_{n,d}(f) \lra \cdots.\]
By construction $\smash{H^{E_k}_{n,d}(\gE_k(S^{1,0}_{\F_p}))}$ vanishes unless $(n,d) = (1,0)$, in which case it is $\bF_p$. As $f_*\colon H_{1,0}(S^{1,0}_{\bF_p}) \m H_{1,0}(\gA)$ is surjective, we obtain that from the above long exact sequence and the Hurewicz theorem of \cite[Corollary 11.12]{e2cellsI} that
\[H^{E_k}_{n,d}(f)  =0 \text{ for $d<n-1$ and } H^{E_k}_{1,0}(f)=H^{E_k}_{2,1}(f)= \ldots = H^{E_k}_{m+1,m}(f)  =0.\]

We now apply the CW approximation result for $\cat{Alg}_{E_k}(\cat{sMod}_{\bF_p}^\bN)$ \cite[Theorem 11.5.3]{e2cellsI}; the result is a factorisation in this category of the map $f$ as
\[\gE_k(S^{1,0}_{\F_p}) \overset{g}\lra \gC \overset{\simeq}\lra \gA,\]
where $g$ is a relative CW object obtained by attaching $E_k$-cells in bidegrees satisfying $d \geq n-1$ for all $n$ and satisfying $d \geq n$ for $n \leq p$. This CW object comes with a skeletal filtration $\mr{sk}(\gC)$, whose associated graded is
\[\gE_k\Big(S^{1,0,0}_{\bF_p} \sigma \vee \bigvee_{i \in I} S^{n_i,d_i,d_i}_{\bF_p} x_i\Big)\]
with $d_i \geq n_i-1$ for all $n_i$ and with $d_i \geq n_i$ for $n_i \leq p$, as a consequence of \cite[Theorem 6.14]{e2cellsI} and the fact that the domain of $g$ is free on a $0$-dimensional generator.

Rectifying and unitalising, and taking the quotient by $\sigma$, we get a spectral sequence (see \cite[Theorem 10.10]{e2cellsI}): 
\[E^1_{n,p,q} = \widetilde{H}_{n,p+q,q} \left(\overline{\gE_k\Big(S^{1,0,0}_{\bF_p} \sigma \vee \bigvee_{i \in I} S^{n_i,d_i,d_i}_{\bF_p} x_i\Big)}/\sigma\right) \Longrightarrow H_{n,a+b}(\overline{\gA}_{\bF_p}/\sigma).\]
We have that
\[H_{*,*,*}\left(\overline{\gE_k\Big(S^{1,0,0}_{\bF_p} \sigma \vee \bigvee_{i \in I} S^{n_i,d_i,d_i}_{\bF_p} x_i\Big)}\right) = W_{k-1}\Big(\bF_p\{\sigma\}[1,0,0] \oplus \bigoplus_{i \in I} \bF_p \{x_i\}[n_i,d_i,d_i]\Big)\]
with right side a free $W_{k-1}$-algebra (see \cite[Section 16.2]{e2cellsI} for a discussion of $W_{k-1}$-algebras). This is a free graded-commutative algebra on iterated Dyer--Lashof, top, and associated operations applied to iterated bracketings of the generators and thus multiplication by $\sigma$ is injective. As a consequence we have an identification
\[E^1_{*,*,*} = \left(W_{k-1}\Big(\bF_p\{\sigma\}[1,0,0] \oplus \bigoplus_{i \in I} \bF_p \{x_i\}[n_i,d_i,d_i]\Big)/\sigma,d^1\right).\]
To finish the proof, we give a lower bound on the non-trivial entries on the $E^1$-page. To do so, we introduce the notion of \emph{slope}: an element in tridegree $(n,p,q)$ has slope $\tfrac{p}{n}$. We claim that all generators have slope $\geq \tfrac{m}{m+1}$. Dyer--Lashof operations and Browder operations preserve or increase slope, so this is clear for Dyer--Lashof operations and Browder operations applied to $x_i$'s. For Dyer--Lashof operations and Browder operations applied to $\sigma$, the assumptions $p \geq 3$ and $k \geq 3$ are needed. The lowest slope generators occur when $p=3$ and are given by the Dyer--Lashof operation $\beta Q^1(\sigma)$ in bidegree $(3,3)$ when $p=3$ and the Browder bracket $[\sigma,\sigma]$ in bidegree $(2,2)$ when $k=3$. This implies the $E^1$-page vanishes in the desired range and hence so does the abutment of the spectral sequence. 
\end{proof}

\subsection{Homological stability for general linear groups after inverting $2$} 
 Recall that in  \autoref{sec:einfty} we defined for each ring $R$ and commutative ring $\bk$ a non-unital $E_\infty$-algebra $\gR_{\bk}$ in $\cat{sMod}_{\bk}^{\bN}$ satisfying
\[H_*(\gR_{\bk}(n)) = \begin{cases} 0 & \text{if $n=0$,} \\
H_*(\GL_n(R);\bk) & \text{if $n>0$,}\end{cases}\]
so that $H_{n,d}(\overline{\gR}_\bk/\sigma) = H_d(\GL_n(R),\GL_{n-1};\bk)$. We now prove \autoref{thm:integers}, which is equivalent to $\gR_{\Z[\quot12]}$ satisfying homological stability of slope 1 for $R$ the integers, the Gaussian integers, or the Eisenstein integers. See \autoref{fig:chart} for an updated chart of the rational homology of $\GL_n(\bZ)$.

\begin{proof}[Proof of \autoref{thm:integers}] Let $R$ be a Euclidean domain such that $BA_n^m$ is spherical for all $n$ and $m$ (e.g.~the integers, the Gaussian integers, or the Eisenstein integers). By \autoref{thm.steinbergcoinv} and the isomorphism
	\[H^{E_1}_{n,d}(\gR_{\Z[\quot12]}) = H_{d-n+1}(\GL_n(R);\St^{E_1}(R) \otimes \Z[\quot12])\]
of \autoref{prop:e1-splitting}, we see that $\gR_{\Z[\quot12]}$ satisfies the hypotheses of \autoref{E3} for $m=\infty$. The claim follows.\end{proof}

We next prove \autoref{thm:Euclidean}, which is equivalent to $\gR_{\Z[\quot12]}$ satisfying homological stability of slope $\tfrac{2}{3}$ when $R$ is a Euclidean domain.

\begin{proof}[Proof of \autoref{thm:Euclidean}]
Let $R$ be a Euclidean domain. By \autoref{S2}, \cite[Theorem 1.1]{Charney}, and the isomorphism 
of \autoref{prop:e1-splitting}, we see that $\gR_{\Z[\quot12]}$ satisfies the hypotheses of \autoref{E3} for $m=2$. The claim follows. 
\end{proof}

\begin{figure}
	\begin{tikzpicture}
		\begin{scope}[scale=.8]
			\def\WW{1.25}
			\clip (-.75,-.75) rectangle ({\WW*10+.5},10.5);
			\draw[thick,Blue] (1,0) -- ({11*\WW},10); 
			\draw (-.5,0)--({\WW*10+0.8},0);
			\draw (0,-1) -- (0,10.5);
			
			\begin{scope}
				\foreach \s in {0,...,16}
				{
					\draw [dotted] (-.5,\s)--({\WW*16+0.25},\s);
					\node [fill=white] at (-.25,\s) [left] {\tiny $\s$};
				}

				\foreach \s in {0,...,16}
				{
					\draw [dotted] ({\WW*\s},-0.5)--({\WW*\s},16.5);
					\node [fill=white] at ({\WW*\s},-.5) {\tiny $\s$};					
				}
			\end{scope}

			\foreach \s in {0,...,10}
			{
				\node [fill=white] at ({\s*\WW},0) {$\bQ$};
				\draw[->,thick] ({\s*\WW+.3},0) -- ({\s*\WW+.9},0);
				\node at ({\s*\WW+.6},0.2) {\tiny $\cong$};
			}
				
			\foreach \s in {5,...,10}
			{
				\node [fill=white] at ({\s*\WW},5) {$\bQ$};
				\draw[->,thick] ({\s*\WW+.3},5) -- ({\s*\WW+.9},5);
				\node at ({\s*\WW+.6},5.2) {\tiny $\cong$};
			}	
		
			\foreach \s in {8,9,10}
			{
				\node [fill=white] at ({8*\WW},\s) {$?$};
			}	
			\foreach \s in {9,10}
			{
				\node [fill=white] at ({9*\WW},\s) {$?$};
			}	
		
			\foreach \s in {7,...,10}
			{
				\draw[->,thick] ({\s*\WW+.3},9) -- ({\s*\WW+.9},9);
			}	
		
			\node at ({7*\WW+.6},9.28) {\tiny inj};
			\node at ({8*\WW+.6},9.28) {\tiny rk ${\geq} 1$};
			\node at ({9*\WW+.6},9.28) {\tiny surj};
			
			\node [fill=white] at ({6*\WW},8) {$\bQ$};
			\node [fill=white] at ({7*\WW},9) {$\bQ$};
			
			\node [fill=white] at ({10*\WW},9) {$\bQ$};
			\node [fill=white] at ({10*\WW},10) {$?$};

			\node [fill=white] at (-.5,-.5) {$\nicefrac{d}{g}$};
			
			\fill[blue!5!white,blend mode=darken] (1,0) -- ({11*\WW},10) -- ({20*\WW},0) -- cycle;
			
			\fill[pattern=north west lines,pattern color=black!20!white,blend mode=darken] ({7.5*\WW},6.6) -- ({7.5*\WW},20) -- ({12*\WW},24) -- ({14*\WW},13) -- cycle;

		\end{scope}
	\end{tikzpicture}
	\caption{The homology groups $H_d(\GL_n(\bZ);\bQ)$. Empty entries are zero groups and question marks denote unknown groups. The stable range is shaded blue, and the unknown unstable range is hatched. These computations are obtained by reduction to the cohomology of $\SL_n(\bZ)$ with coefficients in the trivial and sign representations. For $n=2,3$ these are folklore results (e.g.~\cite[Theorem 42]{SouleNotes}). For $n=4$ one uses that $H_*(\GL_4(\Z);\bQ^\mr{sign})$ was computed by Horozov \cite[Theorem 1.1]{Horozov}. For $n=5,6,7$ one uses \cite[Theorem 7.3]{EVGS}. For the stabilisation maps in degree $5$ and $9$, one uses \cite[Remark 16]{GKT}.}
	\label{fig:chart}
\end{figure}
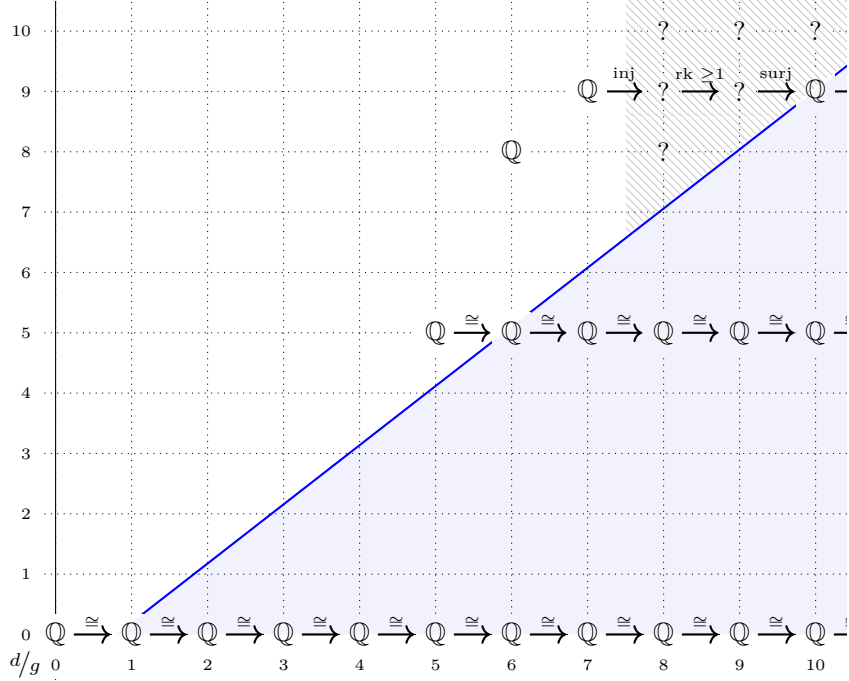

\subsection{Optimality} \label{optimal} \autoref{thm:integers} is optimal in the sense that the slope can not be improved without making the constant term worse. 

\begin{example}[Integers] For the integers, that \autoref{thm:integers} is optimal is a consequence of calculations of Elbaz-Vincent--Gangl--Soul\'{e} that $H_5(\GL_5(\bZ),\GL_{4}(\bZ);\Q) \neq 0$ \cite[Theorem 7.3]{EVGS} (see \autoref{fig:chart}). It is necessary to invert 2 in the coefficients by \cite[Remark 6.6]{e2cellsIII}.\end{example}

\begin{example}[Gaussian integers and Eisenstein integers] For the Gaussian integers or Eisenstein integers, that \autoref{thm:integers} is optimal can be seen more easily. For $R$ the ring of integers in a quadratic imaginary number field, the classifying space of $\GL_2(R)$ is homotopy equivalent to a non-compact hyperbolic $3$-dimensional orbifold and hence $H_3(\GL_2(R);\Q) = 0$. But Borel showed that $\mr{colim}_{n \to \infty} H_3(\GL_\infty(R);\Q) \cong \Q$ for such rings \cite[Section 11]{Bor} so $H_3(\GL_2(R);\Q)$ does not surject onto the stable homology. See also the work of Dutour-Sikiri\'{c}--Gangl--Gunnells--Hanke--Sch\"{u}rmann--Yasaki \cite{DGGJSY}.\end{example}
	
One could also ask about improved homological stability for arithmetic groups. It follows from the work of Borel \cite{Bor} that if $R$ is a number ring and $\Gamma \subseteq \GL_n(R)$ is a finite index subgroup, then $H_*(\Gamma;\Q) \m H_*(\GL_n(R);\Q)$ is an isomorphism in the stable range. As consequence, if $\Gamma_n$ is a finite index subgroup of $\GL_n(R)$ and $\Gamma_{n+1}$ is a finite index subgroup of $\GL_{n+1}(R)$, then there is an isomorphism $H_i(\Gamma_n;\Q) \cong H_i(\Gamma_{n+1};\Q)$ in a stable range. One might ask whether for $R$ the integers, Gaussian integers, or Eisenstein integers, this range is as in \autoref{thm:integers}. This can be not true the case without allowing offsets:

\begin{example}[Congruence subgroups of $\GL_2(\bZ)$]
For example, for $p \geq 7$ the level-$p$ principal congruence subgroups of $\GL_2(\bZ)$ are nontrivial free groups and hence have nontrivial rational first homology group \cite[Theorem 8]{Gunning}. On the other hand, $\GL_1(\Z) \cong \Z/2$ so all of its subgroups have trivial first rational homology group. In particular, for many finite index subgroups $\Gamma_2 \subseteq \GL_2(\Z)$, and subgroups $\Gamma_1 \subseteq \GL_1(\Z) \cap \Gamma_2$, $H_1(\Gamma_2,\Gamma_1;\Q) \neq 0$.
\end{example}

In the sense of homological stability, the full general linear groups $\GL_n(R)$ for $R$ the integers, Gaussian integers, or Eisenstein integers thus seem to be exceptional among the arithmetic subgroups of $\GL_n(R)$. It would be interesting to see whether \autoref{thm:integers} can be recovered through the automorphic approach to the cohomology of arithmetic groups.

\section{Proof of \autoref{thm:integersf2}} In this section we prove \autoref{thm:integersf2} by combining a general homological stability result with a computation in \cite{e2cellsIII}.

\subsection{Cell attachments and homological stability} In this subsection, we prove a lemma which informally says that if you attach an $E_\infty$-cell of high enough slope, it cannot create homological stability that was not already there. The following is a variant of arguments appearing in \cite[Theorem A]{e2cellsIII} and \cite[Theorem 9.4]{e2cellsIV}. 

\begin{lemma} \label{AttachHigh}
Let $\bk$ be a PID and let $\gA \in \cat{Alg}_{E_\infty}(\cat{sMod}_{\bk}^{\bN})$. Assume $H_{*,0}(\overline{\gA}) \cong \bk[\sigma]$ with $\sigma \in H_{1,0}(\gA)$. Take $\gB=\gA \cup^{E_\infty}_\alpha D^{x,y}_{\bk} \beta$ with $y \geq \quot{m}{m+1}x+1$. Then if
\[H_{n,d}(\overline{\gB}/\sigma) = 0 \qquad \text{for $d \leq \quot{m}{m+1}(n-1)$},\]
we also have 
\[H_{n,d}(\overline{\gA}/\sigma) = 0 \qquad \text{for $d \leq \quot{m}{m+1}( n-1)$}.\]
\end{lemma}

\begin{proof}We can lift $\gB$ to a filtered $E_\infty$-algebra $f\gB = 0_* \gA \cup^{E_\infty}_\alpha D^{x,y,1} \beta$. The associated spectral sequence as in \cite[Corollary 10.10]{e2cellsI} has the form: \[E^1_{n,p,q} = \widetilde{H}_{n,p+q,p} \left(\Big(0_*\overline{\gA} \vee^{E_\infty}  S^{x,y,1}_{\bk} \beta\Big)/\sigma \right) \Longrightarrow H_{n,p+q}(\overline{\gB}/\sigma).\] Assume for the purposes of contradiction that $H_{n,d}(\overline{\gA}/\sigma) \neq 0$ for some $d \leq \quot{m}{m+1}( n-1)$; we take such $n$ and $d$ so that $d$ is as small as possible. Then $E^1_{n,0,d} \neq 0$ since $E^1_{n,0,d} \cong H_{n,d}(\overline{\gA}/\sigma)$. We will show there are no differentials into or out of $E^r_{n,0,d} $ for $r \geq 1$ and so $E^1_{n,0,d} \cong E^\infty_{n,0,d}$. Since $E^\infty_{n,0,d}$ is a subquotient of $H_{n,d}(\overline{\gB}/\sigma)$, this will contradict the assumption that  $H_{n,d}(\overline{\gB}/\sigma) = 0$ for $(m+1)d \leq m(n-1)$.

The differentials are given by $d^r \colon E^r_{n,p,q} \to E^r_{n,p-r,q+r-1}$. But since there are no terms in negative filtration we have that $E^1_{n,p,q} = 0$ when $p<0$, there are no non-zero differentials out of the group $E^1_{n,0,d}$. 

We next consider the differentials into $E^1_{n,0,d}$, which have the form $d^r \colon E^r_{n,r,d-r+1} \to E^r_{n,0,d}$. The K\"unneth theorem combined with \cite[Proposition 16.4.1]{e2cellsI} implies that the group $E^1_{n,r,d-r+1}$ is an extension of the two groups
\[ \bigoplus_{n'+n''=n, d'+d''=d+1} H_{n',d'}\left(\overline{\gA}/\sigma \right) \otimes_\bk H_{n'',d'',r}\left(\overline{\gE}_\infty(S^{x,y,1}_{\bk} \beta)  \right),\]
\[\bigoplus_{n'+n''=n, d'+d''=d} \Tor_1^\bk \left( H_{n',d'}\left(\overline{\gA}/\sigma \right), H_{n'',d'',r}\left(\overline{\gE}_\infty(S^{x,y,1}_{\bk} \beta)\right) \right).\] It suffices to prove that both of these groups vanish when $r \geq 1$. In this case, the group $H_{n'',d'',r}(\overline{\gE}_\infty(S^{x,y,1}_{\bk} \beta))$ vanishes whenever $d'' < \tfrac{m}{m+1}n''+1$, so can only be non-trivial when $(m+1)d''\geq mn''+m+1$. By our assumption $d$ is the lowest degree where homological stability for $\gA$ fails and thus $H_{n',d'}(\overline{\gA}/\sigma) = 0$ vanishes for $d' \leq \quot{m}{m+1}(n'-1)$ and $d'<d$, so this group can only be non-trivial for $(m+1)d' > m(n'-1)$. Thus we can only get a non-trivial class when
\begin{align*}(m+1)d&=(m+1)d'+(m+1)d''-(m+1)\\
	&>m(n'-1)+mn''+m+1-m-1\\
	&= m(n'+n''-1) =  m(n-1)\end{align*}
which contradicts our assumption that $(m+1)d \leq m(n-1)$.

A similar computation---in fact, the analogous estimate is easier because we have $d = d'+d''$ rather than $d-1=d'+d''$---shows the $\Tor$-term also vanishes in the desired range, completing the proof.  
\end{proof}

\subsection{Homological stability for a certain small cellular $E_\infty$-algebra}

In this subsection, we describe a ``small'' $E_\infty$-algebra that will serve as our approximation of general linear groups. Using \autoref{AttachHigh}, we will compare this to a non-unital $E_\infty$-algebra used in \cite{e2cellsIII} to study $\GL_n(\F_2)$. 

For the next construction, we will need the following fact: for $\gA \in \cat{Alg}_{E_\infty}(\cat{sMod}_\bZ^\bN)$ and $\sigma \in H_{1,0}(\gA)$, there is an integral lift $Q^1_\bZ(\sigma) \in H_{2,1}(\gA)$ of $Q^1_2(\sigma) \in H_{2,1}(\gA;\bF_2)$.

\begin{definition} 
	Given a ring $R$, we define $\gA_\bZ(R) \in \cat{Alg}_{E_\infty}(\cat{sMod}^\bN_\bZ)$ by \[\gA_\bZ(R) \coloneqq \gE_\infty\left ( S_\bZ^{1,0} \sigma    \oplus \bigoplus_{r \in R^\times }  S_\bZ^{1,1} u_r \right) \cup^{E_\infty}_{\sigma Q^1_\bZ(\sigma)- \sigma^2 u_{-1}} D^{3,2}_\bZ \tau.\]\end{definition}

\begin{proposition} \label{ARstab}
For any ring $R$, $H_{d,n}(\overline{\gA_\bk(R)}/\sigma) \cong 0$  for $d<\quot{2}{3}(n-1)$.
\end{proposition}

\begin{proof}We let $\gB_\bZ(R) \in \cat{Alg}_{E_\infty}(\cat{sMod}_\bZ^\bN)$ be obtained from $\gA_\bZ(R)$ by attaching $E_\infty$-cells $w_r$ in bidegree $(1,2)$ along the generators $u_r$, on each for $r \in R^\times$: \[\gA_\bZ(R) \coloneqq \gA_\bZ(R) \cup^{E_\infty} \bigcup^{E_\infty}_{r \in R^\times} D^{1,2}_\bZ w_r.\] Furthermore, we let $\gB \in \cat{Alg}_{E_\infty}(\cat{sMod}_\bZ^\bN)$ be given by \[\gB \coloneqq \gE_\infty  ( S_\bZ^{1,0} \sigma    )   \cup^{E_\infty}_{\sigma Q^1_\bZ(\sigma)} D^{3,2}_\bZ \tau.\]
	
The composition of the inclusion $\gB \to \gA_\bZ(R)$ and the map $\gA_\bZ(R) \to \gB_\bZ(R)$ is a weak equivalence $\gB \to \gB_\bZ(R)$ since for each $r \in R^\times$ the $E_\infty$-cell $w_r$ cancels $u_r$.  

The proof of \cite[Proposition 6.4]{e2cellsIII} shows that $H_{d,n}(\overline{\gB}/\sigma \otimes \F_2) \cong 0$  for $d<\quot{2}{3}(n-1)$. As $\gB$ is an explicit cellular $E_\infty$-algebra, one can read off a vanishing line for $H_{n,d}^{E_\infty}(\gB)$. The proof of \autoref{E3} then yields that $H_{d,n}(\overline{\gB}/\sigma \otimes \F) = 0$  for $d<\quot{2}{3}(n-1)$ with $\F=\Q$ or $\F_p$ with $p$ odd. This implies $H_{d,n}(\smash{\overline{\gB}/\sigma}) = 0$ in the same range.

We conclude that $H_{d,n}(\gB_\bZ(R)/\sigma) = 0$  for $d<\quot{2}{3}(n-1)$ as well. The claim now follows by \autoref{AttachHigh} taking $\gA$ to be $\gA_\bZ(R)$, $\gB$ to be $\gB_\bZ(R)$, $m$ to be $2$, and $(x,y)$ to be $(1,2)$.
\end{proof}

\subsection{Integral homological stability for general linear groups}  Now suppose that $R$ is a Euclidean domain, so in particular a PID and construct $\gR_\bZ$ as in \autoref{sec:gl-einfty}. To prove integral stability for general linear groups, we need to understand when $\gA_\Z(R)$ is a good approximation for $\gR_{\Z}$. We first begin by describing a map between them. Picking representatives of $\sigma \in H_{1,0}(\gR_\bZ)$ and $u_r \in H_{1,1}(\gR_\bZ)$, we obtain a map
\[\gE_\infty\left ( S_\bZ^{1,0} \sigma   \oplus \bigoplus_{r \in R^\times } S_\bZ^{1,0} u_r \right) \lra \gR_\bZ.\]
We claim $\sigma Q^1_\bZ(\sigma)- \sigma^2 u_{-1}=0 \in H_1(\GL_3(R);\bZ)$. We may verify this in the initial case $R = \bZ$, and since  $\det\colon H_1(\GL_3(\Z);\Z) \to \Z^\times$ is an isomorphism, it suffices to observe that
\[\det \begin{bmatrix} 0 & 1 & 0 \\
	1 & 0 & 0 \\
	0 &0 & 1 \end{bmatrix} = \det \begin{bmatrix} -1 & 0 & 0 \\
	0 & 1 & 0 \\
	0 &0 & 1 \end{bmatrix}.\]
As a consequence, we may extend this to a map 
\[f \colon \gA_\bZ(R) = \gE_\infty\left ( S_\bZ^{1,0} \sigma   \oplus \bigoplus_{r \in R^\times } S_\bZ^{1,0} u_r \right)  \cup^{E_\infty}_{\sigma Q^1_\bZ(\sigma)-\sigma^2 u_{-1}} D^{3,2}_\bZ \tau \lra \gR_\bZ.\]

\begin{lemma} \label{E2cellVanish}
Let $R$ be a Euclidean domain such that the image of $R^\times$ generates $R/2$ as an abelian group. Then $f_*\colon H_{2,1}(\gA_\Z(R)) \m H_{2,1}(\gR_\Z) $ is surjective.
\end{lemma}

\begin{proof}By construction, the image of $f_*$ is equal to the image of the homomorphism $\fS_2 \wr R^\times \to \GL_2(R)$ on abelianisations. This is generated by the matrices
	\[\begin{bmatrix}
		u & 0 \\
		0 & 1
	\end{bmatrix} \text{ for $u \in R^\times$}, \quad \text{and} \quad \begin{bmatrix}
	0 & 1 \\
	1 & 0
\end{bmatrix}.\]
Since $R$ is Euclidean, $\GL_2(R)$ is generated by elementary matrices \cite[\S 2]{Cohn} so it suffices to prove that the class $[E(r)] \in H_1(\GL_2(R);\bZ)$ of each elementary matrix $E(r)\coloneqq\begin{bsmallmatrix}
	1 & r \\
	0 & 1 \end{bsmallmatrix}$ lies in $\mr{im}(f_*) \subset H_1(\GL_2(R);\bZ)$.

We first prove this for $E(1)$. Consider the map $\fS_3 \to \GL_2(R)$ given by permuting the three vectors $e_1,e_2,e_1+e_2$. This induces a map $\bZ/2 \cong H_1(\fS_3;\bZ) \to H_1(\GL_2(R);\bZ)$ with domain generated by any transposition. Thus the transposition of $e_1$ and $e_2$ gives a matrix $\begin{bsmallmatrix}0 & 1 \\ 1 & 0\end{bsmallmatrix}$ and the transposition of $e_1$ and $e_1+e_2$ gives a matrix $\begin{bsmallmatrix} 1 & 1 \\ 0 & -1\end{bsmallmatrix}$, whose classes in $H_1(\GL_2(R);\bZ)$ are equal. That $[E(1)] \in \mr{im}(f_*)$ follows from 
\[\begin{bmatrix} 1 & 1 \\ 0 & 1\end{bmatrix} = \begin{bmatrix} 1 & 0 \\ 0 & -1\end{bmatrix}\begin{bmatrix} 1 & 1 \\ 0 & -1\end{bmatrix},\]

Once we know $[E(1)] \in \mr{im}(f_*)$, that $[E(u)] \in \mr{im}(f_*)$ for $u \in R^\times$ follows from 
\[\begin{bmatrix} 1 & u \\ 0 & 1\end{bmatrix} = \begin{bmatrix} u & 0 \\ 0 & 1\end{bmatrix}\begin{bmatrix} 1 & 1 \\ 0 & 1\end{bmatrix}\begin{bmatrix} u^{-1} & 0 \\ 0 & 1\end{bmatrix}.\]

Finally, $E(a)$ equals $E(-a)$ in $\GL_2(R)^\mr{ab}$ by \[\begin{bmatrix}
	-1 & 0 \\
	0 & 1 \end{bmatrix}  \begin{bmatrix}
	1 & a \\
	0 & 1 \end{bmatrix} \begin{bmatrix}
	-1 & 0 \\
	0 & 1 \end{bmatrix}^{-1} = \begin{bmatrix}
	1 & -a \\
	0 & 1 \end{bmatrix}.\] Since $E(a+b)=E(a)E(b)$, this implies that $[E(2a)] \in H_1(\GL_2(R);\bZ)$ is zero. Now we finally invoke the hypothesis on $R$ to write $r=2a+u_1+\ldots+u_k$ with $u_i$ units. We have that $E(r)=E(2a)E(u_1) \cdots E(u_k)$ and the result follows.\end{proof}

\begin{remark}
One may ask whether $f_*\colon H_{2,1}(\gA_\Z(R)) \m H_{2,1}(\gR_\Z) $ is surjective for all Euclidean domains; this is not the case for $R=\F_2[x]$ \cite{YCor}.
\end{remark}

We now prove \autoref{thm:integersf2}, which is equivalent to $\gR_\bZ$ satisfying homological stability of slope $\tfrac{2}{3}$ when $R$ is a Euclidean domain so that the image of $R^\times$ generates $R/2$.

\begin{proof}[Proof of \autoref{thm:integersf2}] We will prove $H_d(\GL_n(R),\GL_{n-1}(R);\bZ) = H_{d,n}(\overline{\gR_\Z}/\sigma) = 0$  for $d<\quot{2}{3}(n-1)$. This follows by the same argument as in the first paragraph of the proof of \cite[Proposition 6.4]{e2cellsIII} so we will just give the outline. We obtain a vanishing range for $\smash{H_{n,d}^{E_\infty}(f)}$ using \autoref{E2cellVanish}, Charney's connectivity result \cite[Theorem 1.1]{Charney}, and the construction of $\gA_\Z(R)$ as an explicit cellular $E_\infty$-algebra. The claim then follows by considering a cell attachment spectral sequence using a similar argument as in the proof of \autoref{E3}. This uses homological stability for $\gA_\Z(R)$, which was proven in \autoref{ARstab}. 
\end{proof}

\section{Stability with twisted coefficients}

In this section, we improve the ranges for homological stability with so-called polynomial coefficient systems and use this to improve the ranges in Borel vanishing.

\subsection{Stability with polynomial coefficients}

We begin by recalling the notion of a $\VIC(R)$-module appearing in Putman--Sam \cite{PS}.

\begin{definition}
Let $R$ be a ring. Let $\VIC(R)$ be the category with objects finitely generated free $R$-modules with \[\Hom^{\VIC}(M,N)=\{(f,C) \, |\, f\colon M \m N \text{ is injective}, C \subseteq N, \text{ and } \mr{im}(f) \oplus C = N  \}.\] A $\VIC(R)$-module over a ring $\bk$ is a functor $V\colon \VIC(R) \m \Mod_\bk$.
\end{definition}

By evaluating a $\VIC(R)$-module $V$ on $R^n$, we obtain a sequence 
of $\bk[\GL_n(R)]$-modules $V_n$. The standard inclusion $R^n \m R^{n+1}$ together with the standard choice of complement give $\GL_n(R)$-equivariant maps $V_n \m V_{n+1}$. Given a $\VIC(R)$-module $V$, we get induced maps \[H_i(\GL_n(R);V_n) \m H_i(\GL_{n+1}(R);V_{n+1})  \] so it makes sense to talk about homological stability with these twisted coefficients. One of the most common conditions implying twisted homological stability is \emph{polynomiality}, a notion we recall now (see e.g.~\cite{EilenbergMacLaneII, Pirashvili-PolynomialFunctors, RWW,DjamentVespa-WeakPolynomial}).

\begin{definition} \label{DefPolyn}
Let $V$ be a $\VIC(R)$-module. 
\begin{enumerate}[(i)]
	\item Let $\ker(V)$ be the $\VIC(R)$-module with $\ker(V)(M)=\ker(V(M) \m V(M \oplus R))$ and with value on morphisms induced by $V$ in the obvious way. 
	\item Similarly let $\coker(V)$ be the $\VIC(R)$ with $\coker(V)(M)=\coker(V(M) \m V(M \oplus R))$.
	\item We say $V$ is polynomial of degree $-1$ in ranks $>m$ if $V_n=0$ for $n>m$.
	\item We say $V$ is polynomial of degree $r$ in ranks $>m$ if $\ker (V)$ has polynomial degree $-1$ in rank $>m$ and $\coker(V)$ has polynomial degree $\leq r-1$ in ranks $>m-1$.
\end{enumerate}
\end{definition}

The following theorem specialises to give \autoref{thm:twistedStab}.

\begin{theorem}\label{twistedStab}
Let $R$ be a Euclidean domain such that $BA_n^m$ is spherical for all $n$ and $m$ (e.g.~the integers, Gaussian integers, or Eisenstein integers). Let $V$ be a $\VIC(R)$-module over $\Z[\quot12]$ which is polynomial of degree $r$ in ranks $>m$. Then \[H_d(\GL_n(R),\GL_{n-1}(R);V_n,V_{n-1})=0\] for $d<n-\max(r,m)$.
\end{theorem}
\begin{proof}
By \cite[Theorem 4.8]{MPPpoly}, it suffices to prove $H_d(\GL_n(R),\GL_{n-1}(R);\Z[\quot12])=0$ for $d<n$. This was done in the proof of \autoref{thm:integers}.
\end{proof}

Similarly, there are versions of \autoref{thm:Euclidean} and \autoref{thm:integersf2} with twisted coefficients.

\begin{example}\autoref{twistedStab} can be applied to homotopy automorphisms of wedges of spheres. Given a space $X$, let $\hAut(X)$ denote its homotopy equivalences, a group-like topological monoid under composition. The proof of \cite[Theorem 1.6]{MPPpoly} with \autoref{twistedStab} in place of \cite[Theorem 4.10]{MPPpoly} implies that for $k \geq 3$, extension-by-the-identity induces a surjection \[H_d(B\hAut(\vee_{n-1} S^k);\Z[\quot12]) \lra H_d(B\hAut(\vee_n S^k);\Z[\quot12])\] for $d \leq n-1$ and an isomorphism for $d \leq n-2$.\end{example}

\subsection{Borel vanishing}

In this subsection, we improve the stable range of Borel vanishing \cite[Theorem 4.4]{Bor81}. Borel proved that $H^i(\GL_n(R); V) \cong 0$ if $R$ is a number ring (that is, the ring of integers in a number field $K_R$, which we assume comes with a chosen embedding $K_R \hookrightarrow \bC$), $V$ is a finite dimensional $\C[\GL_n (\C)]$-module whose coinvariants vanish, and $i$ is smaller than a function depending on $n$ and $V$. This function is linear of slope $\quot12$ in $n$ and of slope 1 in the highest weights of $V$. We will improve this to slope 1 in $n$ when $R$ is the integers, Gaussian integers, or Eisenstein integers.

Recall that a \emph{polynomial representation} $V$ of $\GL_n(\C)$ (over $\C$) is a group homomorphism $\GL_n(\C) \to \GL (V)$ that is also a polynomial map between varieties. Polynomial representations are semisimple and decompose into irreducible polynomial representations which can described up to isomorphism as follows. Let $\Std_n = \C^n$ denote the standard representation of $\GL_n(\C)$ over $\C$ and let $\Specht_\lambda$ be the Specht module of the symmetric group indexed by the partition $\lambda$. The irreducible polynomial representations of $\GL_n(\C)$ over $\C$ are indexed by partitions $\lambda$ of length at most $n$ and are constructed as
\[ \Std_n^{\otimes r} \otimes_{\C [S_r]} \Specht_\lambda,\]
where $r$ is the size of the partition. For an introduction to this topic, we refer the reader to Green \cite{Green}.

Let $R$ be a subring of $\C$. Polynomial representations can be considered in families for fixed $\lambda$ and varying $n$. Let $\Std\colon \VIC(R) \to\Mod_\C$ be the functor sending a $R$-module $M$ to $\C \otimes_R M$ and which sends a morphism $(f,C)\colon M \m N$ to $f\colon \C \otimes_R M \m \C \otimes_R N$. Define the $\VIC(R)$-module
\[ V(\lambda) = \Std^{\otimes r} \otimes_{\C[S_r]} \Specht_\lambda.\]
This construction assembles irreducible polynomial representations of $\GL_n(R)$ into $\VIC(R)$-modules. 

\begin{proposition}
\label{degofV}

$V(\lambda)$ is a polynomial $\VIC(R)$-module of degree $\leq |\lambda|$.

\end{proposition}

\begin{proof}
Note that $\ker(\Std)_n=0$ for all $n$. We have that $\coker(\Std)_n=\C$ for all $n$ with $\VIC(R)$-morphisms $\coker(\Std)_n \m \coker(\Std)_m$ the identity. Thus, $\coker(\coker(\Std))_n=0$ for all $n$. This means that $\coker(\coker(\Std))$ is polynomial of degree $ -1$, $\coker(\Std)$ is polynomial of degree $ 0$ and $\Std$ is a polynomial of degree $ 1$. Then \cite[Lemma 7.3(b)]{Pa} implies that $\Std^{\otimes r}$ has polynomial degree $\leq r$. Because $- \otimes_{\C[S_r]} \Specht_\lambda$ is an exact functor, it follows that $V(\lambda)$ has polynomial degree $|\lambda|$.
\end{proof}

The following theorem improves Borel's vanishing theorem for nontrivial polynomial representations.

\begin{theorem} \label{BV}
Let $R$ be a number ring such that $BA^m_n$ is spherical for all $n$ and $m$ (e.g.~the integers, Gaussian integers, or Eisenstein integers). Let $V(\lambda)_n$ be the irreducible polynomial representation of $\GL_n(\C)$ indexed by the partition $\lambda$. If $|\lambda|>0$, then \[ H_d(\GL_n(R); V(\lambda)_n) \cong 0 \quad \text{ if $d<  n-|\lambda|$.}\] 
\end{theorem}

\begin{proof}
By \autoref{degofV} and \autoref{twistedStab}, \[ H_d(\GL_n(R); V(\lambda)_n) \cong H_d(\GL_m(R); V(\lambda)_m) \quad \text{ if $d+|\lambda|<n,m$.}\] By \cite[Theorem 4.4]{Bor81}, $H_d(\GL_m(R); V(\lambda)_m) \cong 0$ for $m$ sufficiently large and $|\lambda|>0$. Thus, $H_d(\GL_n(R); V(\lambda)_n) \cong 0 $ for $d<  n-|\lambda|$.
\end{proof}

\bibliographystyle{amsalpha}
\bibliography{../refs}

\vspace{.5cm}

\end{document}